\documentclass[a4paper, 11pt]{article}

\usepackage{fullpage}

\usepackage{amsmath}
\usepackage{amsthm}
\usepackage{amsfonts}
\usepackage{tikz-cd}
\usepackage[utf8x]{inputenc}
\usepackage{enumerate}
\usepackage[english]{babel}
\usepackage{amssymb}
\usepackage{esint}
\usepackage{hyperref}
\usepackage{bm}
\usepackage{bbm}
\usepackage[mathscr]{eucal}

\theoremstyle{plain}
\newtheorem{theorem}{Theorem}[section]
\newtheorem{lemma}[theorem]{Lemma}

\newtheorem{proposition}[theorem]{Proposition}
\theoremstyle{definition}
\newtheorem{definition}[theorem]{Definition}
\theoremstyle{remark}

\newtheorem{remark}[theorem]{Remark}

\numberwithin{equation}{section} 

\DeclareMathOperator{\divv}{div}
\DeclareMathOperator*{\esssup}{ess\:sup}
\DeclareMathOperator*{\essinf}{ess\:inf}
\DeclareMathOperator{\trace}{Tr}
\DeclareMathOperator*{\supp}{supp}

\title{Existence of weak solutions for models of general \\ compressible viscous fluids with linear pressure}
\author{Danica Basari\'{c}}
\date{}

\begin{document}
	\maketitle
	
	\begin{center}
		Technische Universit\"{a}t Berlin \\
		Institute f\"{u}r Mathematik, Stra{\ss}e des 17 Juni 136, 10623 Berlin, Germany 
	\end{center}
	
	\begin{center}
		E-mail address: basaric@math.tu-berlin.de
	\end{center}

	\begin{abstract}
		In this work we will focus on the existence of weak solutions for a system describing a general compressible viscous fluid in the case of the pressure being a linear function of the density and the viscous stress tensor being a non-linear function of the symmetric velocity gradient. More precisely, we will first prove the existence of dissipative solutions and study under which conditions it is possible to guarantee the existence of weak solutions.
	\end{abstract}

	\textbf{Mathematics Subject Classification:} 35A01, 35Q35, 76N10
	
	\vspace{0.2 cm}
	
	\textbf{Keywords:} compressible viscous fluid; weak solution; linear pressure; non-linear viscosity

	\section{Introduction}
	
	The motion of fluids can be modelled through a system of partial differential equations
	\begin{equation} \label{the system}
		 \begin{aligned}
		 \partial_t \varrho + \divv_x (\varrho \textbf{u})&=0,  \\
		 \partial_t(\varrho \textbf{u}) + \divv_x(\varrho \textbf{u} \otimes \textbf{u})+ \nabla_x p &= \divv_x \mathbb{S},
		 \end{aligned}
	\end{equation}
	which can be seen as a mathematical transcription of mainly two physical conservation laws: \textit{conservation of mass} and \textit{conservation of momentum}. For a general non-Newtonian fluid, we can suppose the viscous stress tensor $\mathbb{S}$ to be related to the symmetric velocity gradient 
	\begin{equation*}
		\mathbb{D}_x\textbf{u} = \frac{1}{2}(\nabla_x \textbf{u}+\nabla_x^T \textbf{u})
	\end{equation*}
	through an implicit rheological law of the type
	\begin{equation} \label{constitutive equation}
		\mathbb{S}: \mathbb{D}_x\textbf{u}=F(\mathbb{D}_x\textbf{u}) + F^*(\mathbb{S}),
	\end{equation}
	with $F$ a proper lower semi-continuous function and $F^*$ its conjugate. The physical background of writing the constitutive equation for $\mathbb{S}$ in this form is the fact that $\mathbb{S}$ is monotone in the velocity gradient and vice versa, as clearly explained in the recent survey on a new classification of incompressible fluids by Blechta, M\'{a}lek and Rajagopal \cite{BleMalRaj}. It is worth noticing that choosing
	\begin{equation*}
		F(\mathbb{D}_x\textbf{u}) = \frac{\mu}{2} |\mathbb{D}_x \textbf{u}|^2 +\frac{\lambda}{2} |\divv_x \textbf{u}|^2, \quad \mbox{with } \mu>0, \ \frac{2}{d}+\lambda \geq 0,
	\end{equation*}
	we obtain the compressible Navier-Stokes system. Even though in the latter case there are several results on the existence of global-in-time weak solutions (see e.g. \cite{ChaJinNov}, \cite{Fei}, \cite{Gir}, \cite{Lio}), much less is known for the case where the viscous stress tensor is not a linear function of the velocity gradient: the existence of large-time weak solutions was proved by Feireisl, Liao and M\'{a}lek \cite{FeiLiaMal} in the case where the bulk viscosity $\lambda=\lambda(|\divv_x \textbf{u}|)$ becomes singular for a finite value of $|\divv_x \textbf{u}|$; choosing a linear pressure
	\begin{equation} \label{linear pressure}
		p(\varrho)=a\varrho,
	\end{equation}
	the existence was proved by Mamontov \cite{Mam}, \cite{Mam1} in the context of exponentially growing viscosity coefficients, and by Matu\v{s}\accent23u-Ne\v{c}asov\'{a} and Novotn\'{y} \cite{MatNov}, exploiting the concept of measure-valued solutions. 
	 
	In this work, we are going to study under which hypothesis on the convex potential $F$ appearing in \eqref{constitutive equation} it is possible to guarantee the existence of global-in-time weak solutions for system \eqref{the system},\eqref{constitutive equation} and with a linear pressure of the type \eqref{linear pressure}, cf. Theorem \ref{existence weak solutions}.
	The proof will done via the concept of \textit{dissipative solutions}, satisfying system \eqref{the system} in the distributional sense with an extra defect term in the second equation that we may call \textit{Reynolds stress}. Recently, Abbatiello, Feireisl and Novotn\'{y} \cite{AbbFeiNov} proved the existence of dissipative solutions for system \eqref{the system} with $\mathbb{S}$ satisfying \eqref{constitutive equation} and the isentropic pressure
	\begin{equation*}
		p(\varrho)=a \varrho^{\gamma} \quad \mbox{with } \gamma>1.
	\end{equation*}
	Our goal is to focus on the case $\gamma=1$, for which we will prove the existence of dissipative solutions, cf. Theorem \ref{existence dissipative solutions}. The advantage of relaying on this very weak concept of solution is that they can be easily identified as limits of weakly convergent subsequences of approximate solutions, as we will see in Section \ref{Limit n}. It is worth noticing that our approach represents an alternative and improvement to the ``standard" measure--valued framework applied in this context by Matu\v{s}$\mathring{\mbox{u}}$-Ne\v{c}asov\'{a} and Novotn\'{y} \cite{MatNov}. 
	
	The paper is organized as follows.
	\begin{itemize}
		\item In Section \ref{The system general} we introduce the system we are going to study, fixing the necessary hypothesis on the pressure potential $F$ appearing in \eqref{constitutive equation}.
		\item In Section \ref{Dissipative solution} we provide the definition of dissipative solution for system \eqref{the system}--\eqref{constitutive equation} with the pressure being a linear function of the density, cf. Definition \ref{dissipative solution}.
		\item Section \ref{Existence of dissipative solutions} will be devoted to the proof of the existence of dissipative solutions, cf. Theorem \ref{existence dissipative solutions}. More precisely, we will perform a three-level approximation scheme: addition of artificial viscosity terms in the continuity equation and balance of momentum in order to convert the hyperbolic system into a parabolic one, regularization of the convex potential to make it continuously differentiable, approximation via the Faedo-Galerkin technique and a family of finite-dimensional spaces.
		\item In Section \ref{Existence of weak solutions} we prove the existence of weak solutions for particular choices of the convex potential $F$, cf. Theorem \ref{existence weak solutions}.
		\item In the Appendix \ref{De la Vallee-Poussin criterion} we provide a slightly modified version of the De la Vall\'{e}e--Poussin criterion as we require the stronger condition, with respect to the standard formulation, that the Young function satisfies the $\Delta_2$-condition, cf. Theorem \ref{De la Valee Poussin criterion}, necessary to get the existence of weak solutions.
	\end{itemize}

	\section{The system} \label{The system general}
	
	We are going to study the system described by the following couple of equations
	\begin{align}
		\partial_t \varrho + \divv_x (\varrho \textbf{u})&=0, \label{continuity equation} \\
		\partial_t(\varrho \textbf{u}) + \divv_x(\varrho \textbf{u} \otimes \textbf{u})+ \nabla_x p(\varrho) &= \divv_x \mathbb{S}. \label{balance of momentum}
	\end{align}
	The unknown variables are the density $\varrho=\varrho(t,x)$ and the velocity $\textbf{u}=\textbf{u}(t,x)$ of the fluid, while  the viscous stress tensor $\mathbb{S}$ is assumed to be connected to the symmetric velocity gradient $\mathbb{D}_x\textbf{u}$
	through an implicit rheological law of the type
	\begin{equation} \label{relation viscous stress symmetric gradient}
		\mathbb{S}: \mathbb{D}_x\textbf{u}=F(\mathbb{D}_x\textbf{u}) + F^*(\mathbb{S}).
	\end{equation}
	where, denoting with $\mathbb{R}^{d\times d}_{\textup{sym}}$ the space of $d$-dimensional real symmetric tensors,
	\begin{equation} \label{conditions on F}
		F: \mathbb{R}^{d\times d}_{\textup{sym}} \rightarrow 	[0,\infty) \mbox{ is convex and lower semi-continuous with } F(0)=0,
	\end{equation}
	and $F^*$ is its conjugate. As clearly motivated in \cite{AbbFeiNov}, Section 2.1.2, we will suppose $F$ to satisfy relation
	\begin{equation} \label{relation F and trsce-less part}
		F(\mathbb{D}) \geq \mu \left| \mathbb{D}-\frac{1}{d} \trace[\mathbb{D}]\mathbb{I} \right|^q-c \quad \mbox{for all } \mathbb{D} \in \mathbb{R}^{d\times d}_{\textup{sym}},
	\end{equation}
	for some $\mu >0$, $c>0$ and $q>1$. Notice that condition \eqref{relation viscous stress symmetric gradient} is equivalent in requiring
	\begin{equation*}
		\mathbb{S} \in \partial F(\mathbb{D}\textbf{u}),
	\end{equation*}
	where $\partial$ denotes the subdifferential of a convex function. Furthermore, we will consider a linear barotropic pressure
	\begin{equation} \label{pressure}
		p(\varrho)= a \varrho, \quad a>0;
	\end{equation}
	the pressure potential $P$, satisfying the ODE
	\begin{equation*}
		\varrho P'(\varrho)-P(\varrho)=p(\varrho),
	\end{equation*} 
	will be then of the form
	\begin{equation} \label{pressure potential}
		P(\varrho) = a \ \varrho \log \varrho,
	\end{equation}
	which implies that $P$ is a strictly convex superlinear 	continuous function on $[0,\infty)$. We will study the system on the set 
	\begin{equation*}
		(t,x) \in (0,T)\times \Omega,
	\end{equation*}
	where the time $T>0$ can be chosen arbitrarily large and the physical domain $\Omega \subset \mathbb{R}^d$ is assumed to be bounded and Lipschitz, on the boundary of which we impose the no--slip condition
	\begin{equation} \label{boundary condition}
		\textbf{u}|_{\partial \Omega}=0.
	\end{equation} 
	Finally, we fix the initial conditions
	\begin{equation} \label{initial conditions}
		\varrho(0,\cdot)=\varrho_0, \quad (\varrho \textbf{u})(0,\cdot)= \textbf{m}_0.
	\end{equation}
	
	We conclude this section with the following result, collecting the significant properties of the conjugate function $F^*$.
	
	\begin{proposition}
		Let the function $F$ satisfy conditions \eqref{conditions on F}. Then, its conjugate
		\begin{equation} \label{conditions on F*}
		F^*: \mathbb{R}^{d\times d}_{\textup{sym}} \rightarrow [0,\infty] \mbox{ is convex, lower semi-continuous and superlinear}.
		\end{equation}
	\end{proposition}
	\begin{proof}
		First of all, we recall that $F^*$ is defined for every $\mathbb{A} \in \mathbb{R}^{d\times d}_{\textup{sym}}$ as
		\begin{equation*}
		F^*(\mathbb{A}):= \sup_{\mathbb{B}\in 	\mathbb{R}^{d\times d}_{\textup{sym}}} \{ \mathbb{A}: \mathbb{B}-F(\mathbb{B}) \}.
		\end{equation*}
		The non-negativity of $F^*$ is trivial if $F(0)=0$ since
		\begin{equation*}
		F^*(\mathbb{A}) \geq \mathbb{A}:0- F(0) =0 \quad \mbox{for every } \mathbb{A} \in \mathbb{R}^{d\times d}_{\textup{sym}}.
		\end{equation*}
		It is also well-know that the conjugate is convex and lower semi-continuous as it is the supremum of a family of affine functions. It remains to prove the superlinearity:
		\begin{equation} \label{superlinearity of F star}
		\lim_{|\mathbb{A}|\rightarrow \infty} \frac{F^*(\mathbb{A})}{|\mathbb{A}|}=+\infty.
		\end{equation}
		Let $B_R(0)$ be the ball centred at origin and radius $R>0$; using the fact that for any $\mathbb{A} \in \mathbb{R}^{d\times d}_{\textup{sym}}$
		\begin{equation*}
		\sup_{\mathbb{B} \in B_R(0)} \mathbb{A}: \mathbb{B} = \sup_{\mathbb{B} \in B_R(0)} \{ \mathbb{A}: \mathbb{B}-F(\mathbb{B}) + F(\mathbb{B}) \} \leq F^*(\mathbb{A}) + \sup_{\mathbb{B} \in B_R(0)} F(\mathbb{B}) 
		\end{equation*}
		we have
		\begin{equation*}
		\frac{F^*(\mathbb{A})}{|\mathbb{A}|} \geq \sup_{\substack{0<r\leq R \\ |\mathbb{V}|\leq 1}} \left\{ r \frac{\mathbb{A}}{|\mathbb{A}|} :\mathbb{V} \right\}- \frac{1}{|\mathbb{A}|} \sup_{\mathbb{B} \in B_R(0)} F(\mathbb{B}) \geq R- \frac{c}{|\mathbb{A}|}, 
		\end{equation*}
		where we used the fact that $F(\mathbb{B})$ is finite for any $\mathbb{B} \in \mathbb{R}^{d\times d}_{\textup{sym}}$. We conclude that
		\begin{equation*}
		\liminf_{|\mathbb{A}|\rightarrow \infty} \frac{F^*(\mathbb{A})}{|\mathbb{A}|} \geq R,
		\end{equation*}
		and, since $R$ can be chosen arbitrarily large, we obtain \eqref{superlinearity of F star}.
	\end{proof}

	\section{Dissipative solution} \label{Dissipative solution}
	Following \cite{AbbFeiNov}, we introduce to the concept of \textit{dissipative solutions}, which satisfy the system in the distributional sense but with an extra ``turbulent'' term $\mathfrak{R}$ in the balance of momentum \eqref{balance of momentum} that we may call \textit{Reynolds stress}. As pointed out in \cite{Bas}, Section 4.1.1, in this context, i.e. when the pressure is a linear function of the density, it is only the possible concentrations and/or oscillations in the convective term that contributes to $\mathfrak{R}$. It is worth noticing that when $\mathfrak{R} \equiv 0$, we get the standard notion of weak solution. From now  on, it is better to consider the density $\varrho$ and the momentum $\textbf{m}=\varrho \textbf{u}$ as state variables, since they are at least weakly continuous in time.
	
	\begin{definition} \label{dissipative solution}
		The pair of functions $[\varrho,\textbf{m}]$ constitutes a \textit{dissipative solution} to the problem \eqref{continuity equation}--\eqref{initial conditions} with initial data
		\begin{equation*}
			[\varrho_0, \textbf{m}_0] \in L^1(\Omega)\times L^1(\Omega; \mathbb{R}^d) 
		\end{equation*}
		if the following holds:
		\begin{itemize}
			\item[(i)]  $\varrho \geq 0$ in $(0,T) \times \Omega$ and
			\begin{equation*}
				[\varrho, \textbf{m}] \in C_{ \textup{weak}}([0,T]; L^1(\Omega)) \times  C_{ \textup{weak}}([0,T]; L^1(\Omega;\mathbb{R}^d));
			\end{equation*}
			\item[(ii)] the integral identity
			\begin{equation} \label{weak formulation continuity equation}
				\left[ \int_{\Omega} \varrho \varphi(t,\cdot) \ \textup{d}x \right]_{t=0}^{t=\tau}= \int_{0}^{\tau} \int_{\Omega} [\varrho \partial_t \varphi + \textbf{m}\cdot \nabla_x \varphi] \ \textup{d}x \textup{d}t
			\end{equation}
			holds for any $\tau \in [0,T]$ and any $\varphi \in C_c^1([0,T]\times \overline{\Omega})$, with $\varrho(0,\cdot)=\varrho_0$;
			\item[(iii)] there exist
			\begin{equation*}
				\mathbb{S} \in L^1 (0,T; L^1(\Omega; \mathbb{R}^{d\times d}_{\textup{sym}})) \quad \mbox{and} \quad \mathfrak{R} \in L^{\infty}_{ \textup{weak}}(0,T; \mathcal{M}^+(\overline{\Omega}; \mathbb{R}^{d\times d }_{ \textup{sym}}))
			\end{equation*}
			such that the integral identity
			\begin{equation} \label{weak formulation balance of momentum}
			\begin{aligned}
				\left[ \int_{\Omega} \textbf{m}\cdot \bm{\varphi}(t, \cdot) \ \textup{d}x \right]_{t=0}^{t=\tau} &= \int_{0}^{\tau}\int_{\Omega} \left[ \textbf{m} \cdot \partial_t \bm{\varphi} + \mathbbm{1}_{\varrho>0} \frac{\textbf{m}\otimes \textbf{m}}{\varrho}:\nabla_x \bm{\varphi} +a\varrho \divv_x\bm{\varphi}\right] \ \textup{d}x \textup{d}t \\
				&- \int_{0}^{\tau} \int_{\Omega} \mathbb{S}: \nabla_x \bm{\varphi} \ \textup{d}x \textup{d}t + \int_{0}^{\tau} \int_{\overline{\Omega}} \nabla_x \bm{\varphi} : \textup{d}\mathfrak{R} \ \textup{d}t
			\end{aligned}
			\end{equation}
			holds for any $\tau \in [0,T]$ and any $\bm{\varphi} \in C^1_c([0,T]\times \overline{\Omega}; \mathbb{R}^d)$, $\bm{\varphi}|_{\partial \Omega}=0$, with $\textbf{m}(0,\cdot)=\textbf{m}_0$;
			\item[(iv)] there exists 
			\begin{equation*}
				\textbf{u} \in L^q(0,T; W_0^{1,q}(\Omega; \mathbb{R}^d)) \mbox{ such that } \textbf{m}=\varrho \textbf{u} \mbox{ a.e. in } (0,T)\times \Omega; 
			\end{equation*}
			\item[(v)] there exists a constant $\lambda>0$ such that the energy inequality
			\begin{equation} \label{energy inequality}
				\begin{aligned}
					\int_{\Omega} \left[ \frac{1}{2} \frac{|\textbf{m}|^2}{\varrho} + a\varrho \log \varrho \right](\tau,\cdot) \ \textup{d}x +  \frac{1}{\lambda}\int_{\overline{\Omega}} \textup{d}  \trace[\mathfrak{R}(\tau)] &+ \int_{0}^{\tau} \int_{\Omega} \left[ F(\mathbb{D}\textbf{u})+ F^*(\mathbb{S})\right] \ \textup{d}x \textup{d}t \\
					&\leq \int_{\Omega} \left[ \frac{1}{2} \frac{|\textbf{m}_0|^2}{\varrho_0} + a\varrho_0 \log \varrho_0 \right] \ \textup{d}x 
				\end{aligned}
			\end{equation}
			holds for a.e. $\tau  \in (0,T)$.
		\end{itemize}
	\end{definition}

	\begin{remark} \label{remark on space of measures}
		Here and in the sequel, $\mathcal{M}^+(\overline{\Omega})$ represents the space of all the positive Borel measures on $\overline{\Omega}$, while $\mathcal{M}^+(\overline{\Omega}; \mathbb{R}^{d\times d}_{\rm sym})$ denotes the space of tensor--valued (signed) Borel measures $\mathfrak{R}$ such that 
		\begin{equation*}
		\mathfrak{R}: (\xi \otimes \xi) \in \mathcal{M}^+(\overline{\Omega}),
		\end{equation*}
		for all $\xi \in \mathbb{R}^d$, and with components $\mathfrak{R}_{i,j}=\mathfrak{R}_{j,i}$. $L^{\infty}_{\rm weak}(0,T; \mathcal{M}(\overline{\Omega}))$ denotes the space of all the weak--$*$ measurable mapping $\nu: [0,T] \rightarrow \mathcal{M}(\overline{\Omega})$ such that
		\begin{equation*}
			\esssup_{t\in (0,T)} \|\nu(t,\cdot) \|_{\mathcal{M}(\overline{\Omega})} <\infty,
		\end{equation*}
		which can also be identified as the dual space of $L^1(0,T; C(\overline{\Omega}))$.
	\end{remark}
	
	\section{Existence of dissipative solutions} \label{Existence of dissipative solutions}
	
	As in \cite{AbbFeiNov} Abbatiello, Feireisl and Novotn\'{y} proved the existence of dissipative solutions of system \eqref{continuity equation}--\eqref{initial conditions} with $p(\varrho)=a \varrho^{\gamma}$ and $\gamma>1$, in this section we aim to show existence for $\gamma=1$. We employ an approximation scheme based on
	\begin{itemize}
		\item[(i)] addition of an artificial viscosity term of the type $\varepsilon \Delta_x \varrho$ in the continuity equation \eqref{continuity equation} in order to convert the hyperbolic equation into a parabolic one and thus recover better regularity properties of $\varrho$;
		\item[(ii)] addition of an extra term of the type $\varepsilon \nabla_x \textbf{u} \cdot \nabla_x \varrho$ in the balance of momentum \eqref{balance of momentum} in order to eliminate the extra terms arising in the energy inequality to save the a priori estimates;
		\item[(iii)]  regularization of the convex potential $F$ through convolution with a family of regularizing kernels to make it continuously differentiable.
	\end{itemize}
	
	More precisely, we will study the following system:
	\begin{itemize}
		\item \textbf{continuity equation}
		\begin{equation} \label{approximation continuity equation}
		\partial_t \varrho +\divv_x (\varrho \textbf{u}) = \varepsilon \Delta_x \varrho,
		\end{equation}
		on $(0,T)\times \Omega$, with $\varepsilon>0$, the homogeneous Neumann boundary condition
		\begin{equation} \label{Neumann boundary condition}
		\nabla_x \varrho \cdot \textbf{n} =0 \quad \mbox{on } \partial \Omega,
		\end{equation}
		and the initial condition 
		\begin{equation} \label{approximation initial density}
		\varrho(0,\cdot)= \varrho_{0,n} \quad \mbox{on }\Omega, \quad \varrho_{0,n} \rightarrow \varrho_0 \ \mbox{in } L^1(\Omega) \ \mbox{as } n\rightarrow \infty,
		\end{equation}
		with $\varrho_{0,n} \in C(\overline{\Omega})$, $\varrho_{0,n}>0$ for all $n\in \mathbb{N}$.
		
		\item \textbf{momentum equation} 
		\begin{equation} \label{approximation momentum equation}
		\partial_t(\varrho \textbf{u}) + \divv_x(\varrho \textbf{u}\otimes \textbf{u}) +a\nabla_x \varrho +\varepsilon \nabla_x \textbf{u} \cdot \nabla_x \varrho = \divv_x \mathbb{S}
		\end{equation}
		on $(0,T)\times \Omega$, with $\varepsilon>0$, the no-slip boundary condition
		\begin{equation} \label{boundary condition velocity}
		\textbf{u}|_{\partial \Omega}=0 \quad \mbox{on } \partial \Omega,
		\end{equation}
		and the initial condition 
		\begin{equation} \label{approximation initial momentum}
		(\varrho \textbf{u})(0,\cdot)= \textbf{m}_0 \quad \mbox{on }\Omega.
		\end{equation}
		\item \textbf{convex potential}
		\begin{equation} \label{approximation potential}
		F_{\delta}(\mathbb{D}) = (\xi_{\delta}* F)(\mathbb{D})- \inf_{\mathbb{D}\in \mathbb{R}^{d\times d}_{\textup{sym}}} (\xi_{\delta}* F)
		\end{equation}
		for any $\mathbb{D} \in \mathbb{R}^{d\times d}_{\textup{sym}}$, with $\{ \xi_{\delta} \}_{\delta>0}$ a family of regularizing kernels in $\mathbb{R}^{d\times d}_{\textup{sym}}$, the function $F$ satisfying \eqref{conditions on F}--\eqref{relation F and trsce-less part}, and such that
		\begin{equation} \label{relation viscous stress with F delta}
			\mathbb{S}: F_{\delta}(\mathbb{D}_x \textbf{u}) = F_{\delta}(\mathbb{D}_x \textbf{u})+ F^*_{\delta} (\mathbb{S}).
		\end{equation}
	\end{itemize}
	
	Even if system \eqref{approximation continuity equation}--\eqref{relation viscous stress with F delta} is of parabolic type, we are forced to perform a further approximation known as \textit{Faedo-Galerkin technique}. The reason is that the unknown state variable $\textbf{u}$ appears multiplied by $\varrho$ in \eqref{approximation momentum equation}, which prevents us from applying the already existing results for parabolic systems that can be found in literature. The idea is to consider a family $\{X_n\}_{n\in \mathbb{N}}$ of finite-dimensional spaces $X_n \subset L^2(\Omega; \mathbb{R}^d)$, such that
	\begin{equation*}
	X_n:= { \textup{span}} \{ \textbf{w}_i | \ \textbf{w}_i \in C_c^{\infty}(\Omega; \mathbb{R}^d), \ i=1, \dots, n \},
	\end{equation*}
	where $\textbf{w}_i$ are orthonormal with respect to the standard scalar product in $L^2(\Omega; \mathbb{R}^d)$, and to look for approximated velocities
	\begin{equation*}
	\textbf{u}_n \in C([0,T]; X_n).
	\end{equation*}
	Solvability of the approximated problem will be discussed in the following sections.
	
	\subsection{On the approximated continuity equation}
	
	Given $\textbf{u} \in C([0,T]; X_n)$, let us focus on identifying that unique solution
	\begin{equation*}
		\varrho = \varrho[\textbf{u}]
	\end{equation*}
	of system \eqref{approximation continuity equation}--\eqref{approximation initial density}. As our domain $\Omega$ is merely Lipschitz, we cannot simply repeat the same passages performed for instance by Feireisl \cite{Fei} in the context of the compressible Navier-Stokes system since better regularity for the domain would be required. However, since $X_n$ is finite-dimensional, all the norms on $X_n$ induced by $W^{k,p}$-norms, with $k\in \mathbb{N}$ and $1\leq p\leq \infty$, are equivalent and thus, we deduce that
	\begin{equation*}
	\textbf{u} \in L^{\infty}(0,T; W^{1,\infty}(\Omega; \mathbb{R}^d)),
	\end{equation*}
	and there exist two constants $0<\underline{n}<\overline{n}<\infty$, depending solely on the dimension $n$ of $X_n$, such that for any $t\in [0,T]$
	\begin{equation} \label{connection norm in Xn and W}
	\underline{n} \| \textbf{u}(t, \cdot) \|_{W^{1,\infty}(\Omega)} \leq \| \textbf{u}(t, \cdot) \|_{X_n} \leq \overline{n} \| \textbf{u}(t, \cdot) \|_{W^{1,\infty}(\Omega)}.
	\end{equation}
	
	It is now enough to apply the following result to get the existence of weak solutions and the necessary bounds to recover the existence of the corresponded velocity $\textbf{u}$.
	\begin{lemma} \label{existence approximated densities}
		Let $\Omega \subset \mathbb{R}^d$ be a bounded Lipschitz domain. For any given $\textup{\textbf{u}} \in C([0,T]; X_n)$ and $\varepsilon>0$, there exists a unique weak solution
		\begin{equation*}
		\varrho= \varrho_{\varepsilon, n} \in L^2((0,T); W^{1,2}(\Omega)) \cap C([0,T];L^2(\Omega))
		\end{equation*}
		of system \eqref{approximation continuity equation}--\eqref{approximation initial density} in the sense that the integral identity
		\begin{equation*}
			\left[\int_{\Omega} \varrho \varphi (t, \cdot) \ \textup{d}x\right]_{t=0}^{t=\tau} = \int_{0}^{\tau} \int_{\Omega} (\varrho \partial_t \varphi +\varrho \textup{\textbf{u}} \cdot \nabla_x \varphi -\varepsilon \nabla_x \varrho \cdot \nabla_x \varphi ) \ \textup{d} x,
		\end{equation*}
		holds for any $\tau \in [0,T]$ and any $\varphi \in C^1([0,T]\times \overline{\Omega})$, with $\varrho(0,\cdot)=\varrho_{0,n}$. Moreover,
		\begin{itemize}
			\item[(i)] (bound from above - maximum principle) the weak solution $\varrho$ satisfies
			\begin{equation} \label{bound above density}
			\| \varrho \|_{L^{\infty}((0,\tau) \times \Omega)} \leq \overline{\varrho} \exp \left( \tau \|\divv_x \textup{\textbf{u}}\|_{L^{\infty}((0,T) \times \Omega)} \right),
			\end{equation}
			for any $\tau \in [0,T]$, with
			\begin{equation} \label{maximum approximated initial density}
			\overline{\varrho}:= \max_{\Omega}  \varrho_{0,n};
			\end{equation}
			\item[(ii)] (bound from below) the weak solution $\varrho$ satisfies
			\begin{equation} \label{bound below density}
				\essinf_{(0,\tau) \times \Omega} \varrho(t,x)\geq \underline{\varrho} \exp \left( -\tau \|\divv_x \textup{\textbf{u}}\|_{L^{\infty}((0,T) \times \Omega)}\right),
			\end{equation}
			for any $\tau \in [0,T]$, with
			\begin{equation} \label{minimum approximated initial density}
				\underline{\varrho}:= \min_{\Omega}  \varrho_{0,n};
			\end{equation}
			\item[(iii)] let $\textup{\textbf{u}}_1, \textup{\textbf{u}}_2 \in C([0,T]; X_n)$ be such that
			\begin{equation*}
			\max_{i=1,2} \| \textup{\textbf{u}}_i\|_{L^{\infty}(0,T; W^{1,\infty}(\Omega;\mathbb{R}^d))} \leq K,
			\end{equation*}
			and let $\varrho_i= \varrho[\textup{\textbf{u}}_i]$, $i=1,2$ be the weak solutions of the approximated problem \eqref{approximation continuity equation}--\eqref{approximation initial density} sharing the same initial data $\varrho_{0,n}$ in \eqref{approximation initial density}. Then, for any $\tau \in [0,T]$ we have
			\begin{equation} \label{difference two solution approximates densities}
				\| (\varrho_1- \varrho_2)(\tau, \cdot)\|_{L^2(\Omega)} \leq c_1 \| \textup{\textbf{u}}_1 -\textup{\textbf{u}}_2 \|_{L^{\infty}(0,\tau; W^{1,\infty}(\Omega; \mathbb{R}^d))}
			\end{equation}
			with $c_1=c_1(\varepsilon, \varrho_0, T, K)$.
		\end{itemize}
	\end{lemma}
	\begin{proof}
		For the existence of weak solutions and for (i) see Crippa, Donadello and Spinolo \cite{CriDonSpi}, Lemmas 3.2 and 3.4, for (ii) see Abbatiello, Feireisl and Novotn\'{y} \cite{AbbFeiNov}, Corollary 3.4, and for (iii) see Chang, Jin and Novotn\'{y} \cite{ChaJinNov}, Lemma 4.3 point 3.
	\end{proof}
	
	\subsection{On the approximated balance of momentum}
	
	Let us now turn our attention to the approximated problem \eqref{approximation momentum equation}--\eqref{relation viscous stress with F delta}. Following the same approach performed by Feireisl \cite{Fei}, we will first solve the problem on a time interval $[0,T(n)]$ via a fixed point argument, where $T(n)$ depends on the dimension $n$ of the finite-dimensional space $X_n$. Subsequently we will establish estimates independent of time and iterate the same procedure to finally obtain, after a finite number of steps, our solution $\textbf{u}$ on the whole time interval $[0,T]$.
	
	\subsubsection{Technical preliminaries}
	
	For any $\varrho \in  L^1(\Omega)$, consider the operator $\mathscr{M}[\varrho]: X_n \rightarrow X_n^*$ such that
	\begin{equation} \label{operator M}
	\langle \mathscr{M}[\varrho] \textbf{v}, \textbf{w} \rangle \equiv \int_{\Omega}\varrho \textbf{v} \cdot \textbf{w} \ \textup{d}x,
	\end{equation}
	with $\langle \cdot, \cdot \rangle$ the $L^2$-standard scalar product. In particular, we have
	\begin{equation} \label{norm operator M}
	\|\mathscr{M}[\varrho]\|_{\mathcal{L}(X_n, X_n^*)}=  \sup_{\|\textbf{v} \|_{X_n}, \|\textbf{w}\|_{X_n}\leq 1} |\langle \mathscr{M}[\varrho] \textbf{v}, \textbf{w} \rangle| \leq c(n) \|\varrho\|_{L^1(\Omega)},
	\end{equation}
	It is easy to see that the operator $\mathscr{M}$ is invertible provided $\varrho$ is strictly positive on $\Omega$, and in particular we have
	\begin{equation*}
	\| \mathscr{M}^{-1}[\varrho]\|_{\mathscr{L}(X_n^*; X_n)} = \frac{1}{\inf\{ \|\mathscr{M}[\varrho]\textbf{v}\|_{X_n^*}: \textbf{v}\in X_n, \ \|\textbf{v}\|_{X_n}=1 \}} \leq \frac{c(n)}{\inf_{\Omega}\varrho}. 
	\end{equation*}
	Moreover, the identity
	\begin{equation*}
	\mathscr{M}^{-1}[\varrho_1]-\mathscr{M}^{-1}[\varrho_2]= \mathscr{M}^{-1}[\varrho_2] \left( \mathscr{M}[\varrho_2]-\mathscr{M}[\varrho_1]\right)\mathscr{M}^{-1}[\varrho_1]
	\end{equation*}
	can be used to obtain 
	\begin{equation} \label{difference inverse operator M}
	\left\|\mathscr{M}^{-1}[\varrho_1]-\mathscr{M}^{-1}[\varrho_2]\right\|_{\mathscr{L}(X_n^*; X_n)} \leq c\left(n, \inf_{\Omega} \varrho_1,  \inf_{\Omega} \varrho_2 \right) \left\| \varrho_1-\varrho_2\right\|_{L^1(\Omega)}
	\end{equation}
	for any $\varrho_1, \varrho_2>0$.
	
	\subsubsection{Fixed point argument} 
	
	The approximate velocities $\textbf{u} \in C([0,T]; X_n)$ are looked for to satisfy the integral identity
	\begin{equation} \label{projection momentum equation to finite-dimesional space}
	\begin{aligned}
	\left[ \int_{\Omega}\varrho \textbf{u}(t,\cdot) \cdot \bm{\psi}\ \textup{d}x\right]_{t=0}^{t=\tau} &= \int_{0}^{\tau} \int_{\Omega} \left[ (\varrho \textbf{u} \otimes \textbf{u}) : \nabla_x \bm{\psi} +a\varrho\divv_x \bm{\psi}\right] \textup{d}x\textup{d}t \\
	&- \int_{0}^{\tau} \int_{\Omega} [\partial F_{\delta} (\mathbb{D}_x \textbf{u}): \nabla_x \bm{\psi} +\varepsilon\nabla_x \varrho \cdot \nabla_x \textbf{u} \cdot \bm{\psi}] \ \textup{d}x\textup{d}t
	\end{aligned}
	\end{equation} 
	for any test function $\bm{\psi} \in X_n$ and all $\tau \in [0,T]$. Now, the integral identity \eqref{projection momentum equation to finite-dimesional space} can be rephrased for any $\tau\in [0,T]$ as
	\begin{equation*}
	\langle \mathscr{M}[\varrho(\tau, \cdot)] (\textbf{u}(\tau, \cdot)), \bm{\psi}\rangle = \langle \textbf{m}_0^*, \bm{\psi} \rangle + \langle \int_{0}^{\tau} \mathscr{N} [\varrho(s, \cdot), \textbf{u}(s,\cdot)] \ \textup{d}s, \bm{\psi} \rangle
	\end{equation*}
	with $\mathscr{M}[\varrho]: X_n \rightarrow X_n^*$ defined as in \eqref{operator M}, $\textbf{m}_0^* \in X_n^*$ such that
	\begin{equation*}
	\langle \textbf{m}_0^*, \bm{\psi} \rangle:= \int_{\Omega} \textbf{m}_0 \cdot \bm{\psi}  \ \textup{d}x
	\end{equation*}
	and $\mathscr{N}[\varrho(s, \cdot), \textbf{u}(s, \cdot)] \in X_n^*$ such that
	\begin{align*}
	\langle  \mathscr{N} [\varrho(s, \cdot), \textbf{u}(s, \cdot)], \bm{\psi}\rangle &:=\int_{\Omega} \left[ (\varrho \textbf{u} \otimes \textbf{u}-\partial F_{\delta} (\mathbb{D}_x \textbf{u})) : \nabla_x \bm{\psi} +a\varrho\divv_x \bm{\psi}\right] (s, \cdot) \ \textup{d}x \\
	& - \varepsilon \int_{\Omega} \nabla_x \varrho \cdot \nabla_x \textbf{u} \cdot \bm{\psi} (s, \cdot) \ \textup{d}x.
	\end{align*}
	Here, $\varrho=\varrho[\textbf{u}]$ is the weak solution uniquely determined by $\textbf{u}$ and thus by  Lemma \ref{existence approximated densities}, conditions (i) and (ii), for any $t\in [0,T]$ we have
	\begin{equation} \label{bound from below and above of density}
	0<\underline{\varrho} \exp \left( -t \|\divv_x \textbf{u}\|_{L^{\infty}((0,T) \times \Omega)}\right)\leq \varrho(t,x) \leq \overline{\varrho}\exp \left( t \|\divv_x \textbf{u}\|_{L^{\infty}((0,T) \times \Omega)}\right),
	\end{equation}
	where $\overline{\varrho}$, $\underline{\varrho}$ are defined as in \eqref{maximum approximated initial density}, \eqref{minimum approximated initial density} respectively. In particular, the operator $\mathscr{M}$ is invertible and hence, for any $\tau\in [0,T]$, we can write
	\begin{equation*}
	\textbf{u}(\tau, \cdot) = \mathscr{M}^{-1}[\varrho(\tau, \cdot)] \left(\textbf{m}_0^*+\int_{0}^{\tau} \mathscr{N} [\varrho(s, \cdot), \textbf{u}(s, \cdot)] \ \textup{d}s\right).
	\end{equation*}
	
	For $K$ and $T(n)$ to be fixed, consider a bounded ball $\mathcal{B}(0, \underline{n}K)$ in the space $C([0,T(n)]; X_n)$, with $\underline{n}$ defined as in \eqref{connection norm in Xn and W},
	\begin{equation*}
	\mathcal{B}(0, \underline{n}K):= \left\{ \textbf{v} \in C([0,T(n)]; X_n)\big| \sup_{t\in[0,T(n)]} \| \textbf{v}(t, \cdot)\|_{X_n}\leq \underline{n}K \right\},
	\end{equation*}
	and define a mapping 
	\begin{equation*}
	\mathscr{F}: \mathcal{B}(0, \underline{n}K) \rightarrow C([0,T(n)]; X_n)
	\end{equation*}
	such that for all $\tau\in [0,T(n)]$
	\begin{equation*}
	\mathscr{F}[\textbf{u}](\tau, \cdot):= \mathscr{M}^{-1}[\varrho(\tau, \cdot)] \left(\textbf{m}_0^*+\int_{0}^{\tau} \mathscr{N} [\varrho(s, \cdot), \textbf{u}(s, \cdot)] \ \textup{d}s\right).
	\end{equation*}
	Notice that for every $\textbf{u} \in \mathcal{B}(0, \underline{n}K)$, from \eqref{connection norm in Xn and W} we obtain in particular that for all $t\in [0,T(n)]$
	\begin{equation*}
	\| \textbf{u}(t, \cdot)\|_{W^{1,\infty}(\Omega; \mathbb{R}^d)} \leq K 
	\end{equation*}
	and thus, from \eqref{bound from below and above of density} we obtain that for all $t\in [0,T(n)]$
	\begin{equation*}
	\underline{\varrho} e^{-Kt} \leq \varrho(t,x) \leq \overline{\varrho} e^{Kt}.
	\end{equation*}
	Moreover, it is easy to deduce that for every $\textbf{u} \in \mathcal{B}(0, \underline{n}K)$, $\varrho=\varrho[\textbf{u}]$ and every $t\in [0,T(n)]$
	\begin{equation*}
	\| \mathscr{N}(\varrho(t, \cdot), \textbf{u}(t, \cdot))\|_{X_n^*} \leq c_2(\overline{\varrho}, K, T),
	\end{equation*}
	and for every $\textbf{u}_1, \textbf{u}_2 \in \mathcal{B}(0, \underline{n}K)$, $\varrho_i=\varrho[\textbf{u}_i]$, $i=1,2$ and $t\in [0,T(n)]$, making use of \eqref{difference two solution approximates densities},
	\begin{equation*}
	\| \mathscr{N}(\varrho_1(t, \cdot), \textbf{u}_1(t, \cdot))-\mathscr{N}(\varrho_2(t, \cdot), \textbf{u}_2(t, \cdot))\|_{X_n^*} \leq c_3(\overline{\varrho}, K, T) \| \textbf{u}_1(t, \cdot) - \textbf{u}_2(t, \cdot)\|_{W^{1,\infty}(\Omega; \mathbb{R}^d)}.
	\end{equation*}
	Then, for every $\textbf{u} \in \mathcal{B}(0, \underline{n}K)$, $\varrho=\varrho[\textbf{u}]$ and every $t\in [0,T(n)]$
	\begin{align*}
	\| \mathscr{F}(\textbf{u})(t, \cdot)\|_{X_n} &\leq \| \mathscr{M}^{-1}[\varrho(t, \cdot)]\|_{\mathscr{L}(X_n^*; X_n)} (\| \textbf{m}_0^*\|_{X_n^*}+  \| \mathscr{N}(\varrho(t, \cdot), \textbf{u}(t, \cdot))\|_{X_n^*} \ t) \\
	&\leq \frac{c(n)}{\underline{\varrho}} \ e^{KT(n)} \ \big( \| \textbf{m}_0^*\|_{X_n^*} + c_2 \ T(n)\big),
	\end{align*}
	and for every $\textbf{u}_1, \textbf{u}_2 \in \mathcal{B}(0, \underline{n}K)$, $\varrho_i=\varrho[\textbf{u}_i]$, $i=1,2$ and $t\in [0,T(n)]$,
	\begin{align*}
	\| \mathscr{F}(\textbf{u}_1)(t, \cdot) &-\mathscr{F}(\textbf{u}_2)(t, \cdot) \|_{X_n} \\
	\leq&  \left\| \left( \mathscr{M}^{-1}[\varrho_1(t, \cdot)]- \mathscr{M}^{-1}[\varrho_2(t, \cdot)] \right) \left[\int_{0}^{t} \mathscr{N}(\varrho_1(s, \cdot), \textbf{u}_1(s, \cdot)) \ \textup{d}s \right] \right\|_{X_n} \\
	&+ \left\|  \mathscr{M}^{-1}[\varrho_2(t, \cdot)] \left[\int_{0}^{t} [\mathscr{N}(\varrho_1(s, \cdot), \textbf{u}_1(s, \cdot))- \mathscr{N}(\varrho_2(s, \cdot), \textbf{u}_2(s, \cdot))] \ \textup{d}s \right] \right\|_{X_n} \\
	\leq& \  t \left\| \mathscr{M}^{-1}[\varrho_1(t, \cdot)]- \mathscr{M}^{-1}[\varrho_2(t, \cdot)] \right\|_{\mathscr{L}(X_n^*; X_n)} \| \mathscr{N}(\varrho_1(t, \cdot), \textbf{u}_1(t, \cdot)) \|_{X_n^*}\\
	&+t \left\|\mathscr{M}^{-1}[\varrho_2(t, \cdot)]\right\|_{\mathcal{L}(X_n^*, X_n)} \| \mathscr{N}(\varrho_1(t, \cdot), \textbf{u}_1(t, \cdot))- \mathscr{N}(\varrho_2(t, \cdot), \textbf{u}_2(t, \cdot))\|_{X_n^*}\\
	\leq& \ c(n) \ \frac{c_2}{(\underline{\varrho})^2} \ e^{2Kt}t \ \| \varrho_1(t, \cdot)-\varrho_2(t, \cdot)\|_{L^1(\Omega)} +c(n) \ \frac{c_3}{\underline{\varrho}} \ e^{Kt}t \  \|\textbf{u}_1(t, \cdot)-\textbf{u}_2(t, \cdot)\|_{W^{1,\infty}(\Omega; \mathbb{R}^d)} \\
	\leq& \ c(n) \left( \frac{c_1 c_2}{(\underline{\varrho})^2} + \frac{c_3}{\underline{\varrho}}\right) e^{2Kt}t \ \| \textbf{u}_1(t,\cdot) -\textbf{u}_2(t, \cdot) \|_{W^{1,\infty}(\Omega; \mathbb{R}^d)} \\
	\leq& \ T(n) \ \frac{c(n)}{\underline{n}}\left( \frac{c_1 c_2}{(\underline{\varrho})^2} + \frac{c_3}{\underline{\varrho}}\right) e^{2KT(n)}\ \| \textbf{u}_1(t,\cdot) -\textbf{u}_2(t,\cdot) \|_{X_n}.
	\end{align*}
	Now, taking $K>0$ sufficiently large and $T(n)$ sufficiently small, so that
	\begin{equation*}
	\frac{c(n)}{\underline{\varrho}} \ e^{KT(n)} \ \big( \| \textbf{m}_0^*\|_{X_n^*} + c_2 \ T(n)\big) \leq \underline{n}K,
	\end{equation*}
	and 
	\begin{equation*}
	T(n) \ \frac{c(n)}{\underline{n}}\ \left( \frac{c_1 c_2}{(\underline{\varrho})^2} + \frac{c_3}{\underline{\varrho}}\right) e^{2KT(n)}<1,
	\end{equation*}
	we obtain that $\mathscr{F}$ is a contraction mapping from the closed ball $\mathcal{B}(0, \underline{n} K)$ into itself. From the Banach-Cacciopoli fixed point theorem, we recover that $\mathscr{F}$ admits a unique fixed point $\textbf{u} \in C([0,T(n)]; X_n)$, which in particular solves the integral identity \eqref{projection momentum equation to finite-dimesional space}.
	
	This procedure can be repeated a finite number of times until we reach $T=T(n)$, as long as we have a bound on $\textbf{u}$ independent of $T(n)$. the next section will be dedicated to establish all the necessary estimates.
	
	\subsubsection{Estimates independent of time}
	We start with the \textit{energy estimates}. It follows from \eqref{projection momentum equation to finite-dimesional space} that $\textbf{u}$ is continuously differentiable and, consequently, the integral identity 
	\begin{equation} \label{projection velocities 2} 
	\begin{aligned}
	\int_{\Omega}\partial_t(\varrho \textbf{u}) \cdot \bm{\psi}\ \textup{d}x &= \int_{\Omega} \left[ \varrho \textbf{u} \otimes \textbf{u} : \nabla_x \bm{\psi} +a\varrho\divv_x \bm{\psi}-\partial F_{\delta} (\mathbb{D}_x \textbf{u}): \nabla_x \bm{\psi} \right] \textup{d}x \\
	&- \varepsilon \int_{\Omega} [ \nabla_x \varrho \cdot \nabla_x \textbf{u} \cdot \bm{\psi}] \ \textup{d}x
	\end{aligned}
	\end{equation} 
	holds on $(0,T(n))$ for any $\bm{\psi} \in X_n$, with $\varrho= \varrho[\textbf{u}]$. We recall that in this context the pressure potential $P=P(\varrho)$ satisfies the following identity
	\begin{equation*}
		a \varrho \divv_x \textbf{u} = -\partial_t P(\varrho) -\divv_x(P(\varrho) \textbf{u}) +\varepsilon \ a  (\log \varrho +1) \Delta_x \varrho.
	\end{equation*}
	
	Now, taking $\bm{\psi}=\textbf{u}$ in \eqref{projection velocities 2} and noticing that
	\begin{align*}
	\int_{\Omega}[\partial_t(\varrho \textbf{u}) \cdot\textbf{u} -   \varrho \textbf{u} \otimes \textbf{u} : \nabla_x \textbf{u}] \ \textup{d}x&=\frac{\textup{d}}{\textup{d}t} \int_{\Omega} \frac{1}{2} \varrho |\textbf{u}|^2 \ \textup{d}x +\frac{1}{2} \int_{\Omega} (\partial_t \varrho+ \divv_x(\varrho\textbf{u})) |\textbf{u}|^2 \ \textup{d}x \\
	&=\frac{\textup{d}}{\textup{d}t} \int_{\Omega} \frac{1}{2} \varrho |\textbf{u}|^2 \textup{d}x +\frac{\varepsilon}{2} \int_{\Omega} \Delta_x \varrho |\textbf{u}|^2 \ \textup{d}x,
	\end{align*}
	where, using the boundary condition \eqref{Neumann boundary condition},
	\begin{align*}
	\frac{\varepsilon}{2}\int_{\Omega} |\textbf{u}|^2 \Delta_x \varrho \ \textup{d}x &= \frac{\varepsilon}{2}\int_{\Omega} |\textbf{u}|^2 \divv_x \nabla_x \varrho \ \textup{d}x \\
	&= \frac{\varepsilon}{2}\int_{\Omega} \divv_x ( |\textbf{u}|^2 \nabla_x \varrho) \ \textup{d}x - \frac{\varepsilon}{2}\int_{\Omega} \nabla_x \varrho \cdot \nabla_x \textbf{u} \cdot 2 \textbf{u} \ \textup{d}x \\
	&=\frac{\varepsilon}{2}\int_{\partial\Omega}|\textbf{u}|^2 \nabla_x \varrho \cdot \textbf{n} \ \textup{d}S_x - \varepsilon\int_{\Omega} \nabla_x \varrho \cdot \nabla_x \textbf{u} \cdot \textbf{u} \ \textup{d}x \\
	&= -\varepsilon\int_{\Omega} \nabla_x \varrho \cdot \nabla_x \textbf{u} \cdot \textbf{u} \ \textup{d}x,
	\end{align*}
	and
	\begin{align*}
	\int_{\Omega} (\log \varrho +1) \Delta_x \varrho  \ \textup{d}x&= \int_{\Omega} (\log \varrho +1)  \divv_x \nabla_x \varrho  \ \textup{d}x \\
	&= \int_{\Omega} \divv_x \left[(\log \varrho +1)  \nabla_x\varrho \right] \textup{d}x - \int_{\Omega} \nabla_x (\log \varrho +1)  \cdot \nabla_x \varrho  \ \textup{d}x \\
	&= \int_{\partial \Omega}(\log \varrho +1) \nabla_x\varrho \cdot \textbf{n} \ \textup{d}S_x -\int_{\Omega} \frac{d}{d\varrho}(\log \varrho +1)|\nabla_x \varrho|^2  \ \textup{d}x \\
	&=- \int_{\Omega} \frac{1}{\varrho} |\nabla_x \varrho|^2  \ \textup{d}x =- \int_{\Omega} P''(\varrho) |\nabla_x \varrho|^2  \ \textup{d}x ,
	\end{align*}
	we finally obtain
	\begin{equation} \label{first approximate energy equality}
	\frac{\textup{d}}{\textup{d}t} \int_{\Omega}\left[ \frac{1}{2} \varrho |\textbf{u}|^2 +P(\varrho)\right] \textup{d}x= -\varepsilon \int_{\Omega} P''(\varrho) |\nabla_x \varrho|^2 \ \textup{d}x - \int_{\Omega} \partial F_{\delta} (\mathbb{D}_x \textbf{u}): \nabla_x \textbf{u} \ \textup{d}x.
	\end{equation}
	Note that we got rid of the integral 				$\frac{\varepsilon}{2}\int_{\Omega} |\textbf{u}_n|^2 \Delta_x \varrho_n \textup{d}x$ thanks to the extra term $\varepsilon \nabla_x \varrho_n \cdot \nabla_x \textbf{u}_n$ in \eqref{approximation momentum equation}. Since all the quantities involved are at least continuous in time, we may integrate \eqref{first approximate energy equality} over $(0,\tau)$ in order to get the following energy equality
	\begin{equation} \label{second approximate energy equality}
	\begin{aligned}
	\int_{\Omega}\left[ \frac{1}{2} \varrho |\textbf{u}|^2 +P(\varrho)\right](\tau, \cdot) \ \textup{d}x &+ \int_{0}^{\tau} \int_{\Omega}  \left[\partial F_{\delta} (\mathbb{D}_x \textbf{u}): \nabla_x \textbf{u} + \varepsilon P''(\varrho) |\nabla_x \varrho|^2\right] \textup{d}x \textup{d}t \\
	&= \int_{\Omega} \left[\frac{1}{2} \frac{|\textbf{m}_0|^2}{\varrho_{0,n}}+ P(\varrho_{0,n})\right]\textup{d}x,
	\end{aligned}
	\end{equation}
	for any time $\tau \in [0,T(n)]$. In particular, if we suppose the initial value of the (modified) total energy
	\begin{equation} \label{initial energies independent of n}
		 \int_{\Omega} \left[\frac{1}{2} \frac{|\textbf{m}_0|^2}{\varrho_{0,n}} +P(\varrho_{0,n})\right] \textup{d}x \leq \overline{E}
	\end{equation}
	where the constant $\overline{E}$ is independent of $n>0$, the term on the right-hand side of \eqref{second approximate energy equality} is bounded.
	
	Now, the following result, collecting all the significant properties of the regularized potential $F_{\delta}$, is needed.
	\begin{proposition}
		For every fixed $\delta>0$ and $F$ satisfying hypothesis \eqref{conditions on F}--\eqref{relation F and trsce-less part}, the function $F_{\delta}$ defined in \eqref{approximation potential} is convex, non-negative, infinitely differentiable and such that
		\begin{equation} \label{relation F delta trace-less part}
		F_{\delta}(\mathbb{D})\geq \nu \left| \mathbb{D}-\frac{1}{d}\trace[\mathbb{D}]\mathbb{I} \right|^q-c \quad \mbox{for all } \mathbb{D} \in \mathbb{R}_{\textup{sym}}^{d\times d}
		\end{equation}
		with $\nu>0$, $c>0$, $q>1$ independent of $\delta$.
	\end{proposition}
	\begin{proof}
		For every fixed $\delta>0$, the non-negativity of $F_{\delta}$ is trivial while smoothness follows from the fact that each derivative can be transferred to the mollifiers $\xi_{\delta}$. Moreover, for every $\mathbb{A}, \mathbb{B} \in \mathbb{R}_{\textup{sym}}^{d\times d}$ and every $t\in [0,1]$, denoting
		\begin{equation*}
		C_1:= \inf_{\mathbb{D}\in \mathbb{R}^{d\times d}_{\textup{sym}}} \int_{\mathbb{R}^{d\times d}_{\textup{sym}}} \xi_{\delta}(|\mathbb{D}-\mathbb{Z}|) F(\mathbb{Z}) \ \textup{d}\mathbb{Z}
		\end{equation*}
		we have
		\begin{align*}
		F_{\delta}&(t\mathbb{A}+(1-t)\mathbb{B}) \\
		&= \int_{\mathbb{R}^{d\times d}_{\textup{sym}}} F\big(t(\mathbb{A}+\mathbb{Z})+(1-t)(\mathbb{B}+\mathbb{Z})\big) \  \xi_{\delta}(|\mathbb{Z}|)  \ \textup{d} \mathbb{Z} + tC_1 -(1-t)C_1 \\
		&\leq t \left(\int_{\mathbb{R}^{d\times d}_{\textup{sym}}} F(\mathbb{A}+\mathbb{Z}) \  \xi_{\delta}(|\mathbb{Z}|)  \ \textup{d} \mathbb{Z}+C_1\right) +(1-t) \left(\int_{\mathbb{R}^{d\times d}_{\textup{sym}}} F(\mathbb{B}+\mathbb{Z}) \  \xi_{\delta}(|\mathbb{Z}|)  \ \textup{d} \mathbb{Z}+C_1\right) \\
		&= tF_{\delta} (\mathbb{A}) +(1-t) F_{\delta}(\mathbb{B}),
		\end{align*}	
		where we have simply summed and subtracted terms $t\mathbb{Z}$, $tC_1$ in the second line and used the convexity of $F$ in the third line. In particular, we get that for every fixed $\delta>0$, $F_{\delta}: \mathbb{R}^{d\times d}_{\textup{sym}} \rightarrow [0,\infty)$ is convex. 
		
		Let now $\mathbb{D} \in \mathbb{R}^{d\times d}_{\textup{sym}}$ be fixed. From \eqref{relation F and trsce-less part}, we have
		\begin{align*}
		F_{\delta}(\mathbb{D}) &= \int_{\mathbb{R}^{d\times d}_{\textup{sym}}} F(\mathbb{D}-\mathbb{Z}) \xi_{\delta}(|\mathbb{Z}|)  \ \textup{d}\mathbb{Z} - C_1 \\
		&\geq \mu \int_{\mathbb{R}^{d\times d}_{\textup{sym}}} \left| \left( \mathbb{D}-\frac{1}{d} \trace[\mathbb{D}] \mathbb{I} \right) -  \left( \mathbb{Z}-\frac{1}{d} \trace[\mathbb{Z}] \mathbb{I} \right)\right|^q \xi_{\delta}(|\mathbb{Z}|)  \ \textup{d}\mathbb{Z} - C_1.
		\end{align*}
		Applying Minkowski's inequality, we get
		\begin{align*}
		\int_{\mathbb{R}^{d\times d}_{\textup{sym}}} &\left| \left( \mathbb{D}-\frac{1}{d} \trace[\mathbb{D}] \mathbb{I} \right) -  \left( \mathbb{Z}-\frac{1}{d} \trace[\mathbb{Z}] \mathbb{I} \right)\right|^q \xi_{\delta}(|\mathbb{Z}|)  \ \textup{d}\mathbb{Z} \\
		\geq& \ \left[ \left( \int_{\mathbb{R}^{d\times d}_{\textup{sym}}} \left| \mathbb{D}-\frac{1}{d} \trace[\mathbb{D}] \mathbb{I}  \right|^q \xi_{\delta}(|\mathbb{Z}|)  \textup{d} \mathbb{Z} \right)^{\frac{1}{q}} - \left(\int_{\mathbb{R}^{d\times d}_{\textup{sym}}} \left| \mathbb{Z}-\frac{1}{d} \trace[\mathbb{Z}] \mathbb{I}  \right|^q \xi_{\delta}(|\mathbb{Z}|) \textup{d} \mathbb{Z}\right)^{\frac{1}{q}}\right]^q;
		\end{align*}
		recalling that for any $\delta>0$ sufficiently small $\supp \xi_{\delta} \subset K$ with $K\subset \mathbb{R}^{d\times d}_{\textup{sym}}$ a compact set and that for any $\delta>0$
		\begin{equation*}
		\int_{\mathbb{R}^{d\times d}_{\textup{sym}}} \xi_{\delta}(|\mathbb{Z}|) \ \textup{d}\mathbb{Z}= \frac{1}{\delta^d} \int_{\mathbb{R}^{d\times d}_{\textup{sym}}}\xi\left( \frac{|\mathbb{Z}|}{\delta} \right) \textup{d}\mathbb{Z}= \int_{\mathbb{R}^{d\times d}_{\textup{sym}}} \xi(|\mathbb{Z}|) \ \textup{d}\mathbb{Z} =1,
		\end{equation*}
		we obtain the following inequality
		\begin{equation*}
		F_{\delta}(\mathbb{D}) \geq   \mu \left[ \left| \mathbb{D}-\frac{1}{d} \trace[\mathbb{D}] \mathbb{I}  \right|- \left( \sup_{\mathbb{Z}\in K} \left| \mathbb{Z}-\frac{1}{d} \trace[\mathbb{Z}] \mathbb{I}  \right|^q\right)^{\frac{1}{q}}  \right]^q-C_1.
		\end{equation*}
		Now, for every fixed $q>1$ and constant $c_1>0$, there exist $\alpha=\alpha(q, c_1) \in (0,1)$ and $c_2=c_2(q, c_1)>0$ such that
		\begin{equation*}
		(y-c_1)^q \geq \alpha  y^q- c_2 \quad \mbox{for any } y\geq 0;
		\end{equation*}
		in particular, we get that for all $\mathbb{D}\in \mathbb{R}^{d\times d}_{\textup{sym}}$
		\begin{equation*}
		F_{\delta}(\mathbb{D}) \geq \mu \alpha \left| \mathbb{D}-\frac{1}{d} \trace[\mathbb{D}] \mathbb{I}  \right|^q - (C_1+C_2) 
		\end{equation*}
		and thus \eqref{relation F delta trace-less part} holds choosing $\nu=\mu \alpha$ and $c=C_1+C_2$.
	\end{proof}

	From \eqref{second approximate energy equality} and \eqref{initial energies independent of n} we can deduce that
	\begin{equation*}
		\| F_{\delta} (\mathbb{D}_x \textbf{u}_n) \|_{L^1((0,T)\times \Omega)} \leq c(\overline{E})
	\end{equation*}
	which, from \eqref{relation F delta trace-less part}, implies
	\begin{equation*}
	\left\| \mathbb{D}_x\textbf{u}_n -\frac{1}{d} (\divv_x \textbf{u}_n) \mathbb{I}\right\|_{L^q((0,T)\times \Omega;\mathbb{R}^{d\times d})} \leq c(\overline{E}).
	\end{equation*}
	The previous inequality combined with the $L^q$-version of the trace-free Korn's inequality, see \cite{BreCiaDie}, Theorem 3.1, gives
	\begin{equation*}
		\| \nabla_x \textbf{u}_n\|_{L^q((0,T)\times \Omega; \mathbb{R}^{d\times d})} \leq c(\overline{E});
	\end{equation*}
	the standard Poincar\'{e} inequality ensures then
	\begin{equation*}
		\textbf{u} \mbox{ to be bounded in } L^q(0,T(n); W_0^{1,q}(\Omega; \mathbb{R}^d))
	\end{equation*}
	by a constant which is independent of $n$ and $T(n)\leq T$. Since all norms are equivalent in $X_n$, this implies that 
	\begin{equation*}
		\textbf{u} \mbox{ is bounded in } L^q(0,T(n); W^{1,\infty}(\Omega; \mathbb{R}^d));
	\end{equation*}
	in particular, by virtue of \eqref{bound above density} and \eqref{bound below density}, the density $\varrho= \varrho[\textbf{u}]$ is bounded from below and above by constants independent of $T(n)\leq T$. Since $\varrho$ is bounded from below, one can use \eqref{second approximate energy equality} to easily deduce uniform boundedness in $t$ of $\textbf{u}$ in the space $L^2(\Omega; \mathbb{R}^d)$. Consequently, the functions $\textbf{u}(t, \cdot)$ remain bounded in $X_n$ for any $t$ independently of $T(n)\leq T$. Thus we are allowed to iterate the previous local existence result to construct a solution defined on the whole time interval $[0,T]$. 
	
	Summarizing, so far we proved the following result.
	
	\begin{lemma} \label{existence in delta epsilon n}
		For every fixed $\delta>0$, $\varepsilon>0$, $n\in \mathbb{N}$, and any $\varrho_{0,n} \in C(\overline{\Omega})$ such that
		\begin{equation*}
			\int_{\Omega} \left[\frac{1}{2} \frac{|\textup{\textbf{m}}_0|^2}{\varrho_{0,n}} +P(\varrho_{0,n})\right] \textup{d}x \leq \overline{E},
		\end{equation*}
		where the constant $\overline{E}$ is independent of $n$, there exist
		\begin{align*}
			\varrho=\varrho_{\delta, \varepsilon, n} &\in L^2((0,T); W^{1,2}(\Omega)) \cap C([0,T];L^2(\Omega)), \\
			\textup{\textbf{u}}=\textup{\textbf{u}}_{\delta, \varepsilon, n} &\in C([0,T]; X_n),
		\end{align*}
		such that
		\begin{itemize}
			\item[(i)] the integral identity
			\begin{equation*}
			\left[\int_{\Omega} \varrho \varphi (t, \cdot) \ \textup{d}x\right]_{t=0}^{t=\tau} = \int_{0}^{\tau} \int_{\Omega} (\varrho \partial_t \varphi +\varrho \textup{\textbf{u}} \cdot \nabla_x \varphi -\varepsilon \nabla_x \varrho \cdot \nabla_x \varphi ) \ \textup{d} x
			\end{equation*}
			holds for any $\tau \in [0,T]$ and any $\varphi \in C^1([0,T]\times \overline{\Omega})$, with $\varrho(0, \cdot)=\varrho_{0,n}$;
			\item[(ii)] the integral identity
			\begin{equation*}
			\begin{aligned}
			\left[\int_{\Omega}\varrho \textup{\textbf{u}} \cdot \bm{\varphi}(t, \cdot)\ \textup{d}x\right]_{t=0}^{t=\tau}  &= \int_{0}^{\tau}\int_{\Omega} \left[ \varrho\textup{\textbf{u}} \cdot \partial_t\bm{\varphi}+(\varrho \textup{\textbf{u}} \otimes \textup{\textbf{u}}) : \nabla_x \bm{\varphi} +a\varrho\divv_x \bm{\varphi} \right] \textup{d}x\textup{d}t \\
			&-\int_{0}^{\tau}\int_{\Omega}\partial F_{\delta} (\mathbb{D}_x \textup{\textbf{u}}): \nabla_x \bm{\varphi} \ \textup{d}x\textup{d}t -\varepsilon \int_{0}^{\tau}\int_{\Omega} \nabla_x \varrho\cdot \nabla_x \textup{\textbf{u}} \cdot \bm{\varphi} \ \textup{d}x\textup{d}t
			\end{aligned}
			\end{equation*} 
			holds for any $\tau \in [0,T]$ and any $\bm{\varphi} \in C^1([0,T]; X_n)$, with $(\varrho \textup{\textbf{u}})(0, \cdot)= \textup{\textbf{m}}_0$;
			\item[(iii)] the integral equality
			\begin{align*}
				\int_{\Omega}\left[ \frac{1}{2} \varrho |\textbf{u}|^2 +P(\varrho)\right](\tau, \cdot) \ \textup{d}x &+ \int_{0}^{\tau} \int_{\Omega}  \partial F_{\delta} (\mathbb{D}_x \textbf{u}): \nabla_x \textbf{u} \  \textup{d}x \textup{d}t + \varepsilon \int_{0}^{\tau} \int_{\Omega} P''(\varrho) |\nabla_x \varrho|^2 \textup{d}x \textup{d}t \\
				&= \int_{\Omega} \left[\frac{1}{2} \frac{|\textbf{m}_0|^2}{\varrho_{0,n}}+ P(\varrho_{0,n})\right]\textup{d}x
			\end{align*}
			holds for any time $\tau \in [0,T]$.
		\end{itemize}
	\end{lemma}
	
	\subsection{Limit $\delta \rightarrow 0$}
	Let now $\varepsilon>0$ and $n \in \mathbb{N}$ be fixed, and let $\{ \varrho_{\delta}, \textbf{u}_{\delta} \}_{\delta>0}$ be the family of weak solutions to problem \eqref{approximation continuity equation}--\eqref{relation viscous stress with F delta} as in Lemma \ref{existence in delta epsilon n}. Proceeding as before, we can deduce that
	\begin{equation*}
	\{\textbf{u}_{\delta}\}_{\delta>0} \mbox{ is unifrmly bounded in } L^q(0,T; W_0^{1,q}(\Omega; \mathbb{R}^d)).
	\end{equation*}
	As $n$ is fixed and all norms are equivalent on the finite-dimensional space $X_n$, we get that
	\begin{equation*}
	\{ \nabla_x \textbf{u}_{\delta} \}_{\delta>0} \mbox{ is unifrmly bounded in }  L^{\infty}((0,T)\times \Omega; \mathbb{R}^{d\times d}),
	\end{equation*}
	and therefore, we are ready to perform the limit $\delta \rightarrow 0$. Accordingly, we obtain the following result.
	\begin{lemma} \label{limit delto to zero}
		For every fixed $\varepsilon>0$, $n\in \mathbb{N}$, and any $\varrho_{0,n} \in C(\overline{\Omega})$ such that
		\begin{equation*}
		\int_{\Omega} \left[\frac{1}{2} \frac{|\textup{\textbf{m}}_0|^2}{\varrho_{0,n}} +P(\varrho_{0,n})\right] \textup{d}x \leq \overline{E},
		\end{equation*}
		where the constant $\overline{E}$ is independent of $n$, there exist
		\begin{align*}
			\varrho&=\varrho_{\varepsilon, n} \in L^2((0,T); W^{1,2}(\Omega)) \cap C([0,T];L^2(\Omega)), \\
		\textup{\textbf{u}}&=\textup{\textbf{u}}_{\varepsilon, n} \in C([0,T]; X_n),
		\end{align*}
		such that
		\begin{itemize}
			\item[(i)] the integral identity
			\begin{equation} \label{weak formulation continuity equation epsilon}
			\left[\int_{\Omega} \varrho \varphi (t, \cdot) \ \textup{d}x\right]_{t=0}^{t=\tau} = \int_{0}^{\tau} \int_{\Omega} (\varrho \partial_t \varphi +\varrho \textup{\textbf{u}} \cdot \nabla_x \varphi -\varepsilon \nabla_x \varrho \cdot \nabla_x \varphi ) \ \textup{d} x
			\end{equation}
			holds for any $\tau \in [0,T]$ and any $\varphi \in C^1([0,T]\times \overline{\Omega})$, with $\varrho(0, \cdot)=\varrho_{0,n}$;
			\item[(ii)] there exists
			\begin{equation*}
			\mathbb{S}=\mathbb{S}_{\varepsilon,n} \in L^{\infty}((0,T)\times \Omega; \mathbb{R}^{d\times d}_{\textup{sym}})
			\end{equation*}
			such that the integral identity
			\begin{equation} \label{weak formulation momentum equation epsilon}
			\begin{aligned}
			\left[\int_{\Omega}\varrho \textup{\textbf{u}} \cdot \bm{\varphi}(t, \cdot)\ \textup{d}x\right]_{t=0}^{t=\tau}  &= \int_{0}^{\tau}\int_{\Omega} \left[ \varrho\textup{\textbf{u}} \cdot \partial_t\bm{\varphi}+(\varrho \textup{\textbf{u}} \otimes \textup{\textbf{u}}) : \nabla_x \bm{\varphi} +a\varrho\divv_x \bm{\varphi} \right] \textup{d}x\textup{d}t \\
			&-\int_{0}^{\tau}\int_{\Omega} \mathbb{S}: \nabla_x \bm{\varphi} \ \textup{d}x\textup{d}t -\varepsilon \int_{0}^{\tau}\int_{\Omega} \nabla_x \varrho\cdot \nabla_x \textup{\textbf{u}} \cdot \bm{\varphi} \ \textup{d}x\textup{d}t
			\end{aligned}
			\end{equation} 
			holds for any $\tau \in [0,T]$ and any $\bm{\varphi} \in C^1([0,T]; X_n)$, with $(\varrho \textup{\textbf{u}})(0, \cdot)= \textup{\textbf{m}}_0$;
			\item[(iii)] the integral inequality
			\begin{equation} \label{energy inequality in epsilon and n}
			\begin{aligned}
				\int_{\Omega}\left[ \frac{1}{2} \varrho |\textbf{u}|^2 +P(\varrho)\right](\tau, \cdot) \ \textup{d}x &+ \int_{0}^{\tau} \int_{\Omega}  [F(\mathbb{D}_x \textup{\textbf{u}})+F^*(\mathbb{S})] \  \textup{d}x \textup{d}t + \varepsilon \int_{0}^{\tau} \int_{\Omega} P''(\varrho) |\nabla_x \varrho|^2 \textup{d}x \textup{d}t \\
				&\leq \int_{\Omega} \left[\frac{1}{2} \frac{|\textbf{m}_0|^2}{\varrho_{0,n}}+ P(\varrho_{0,n})\right]\textup{d}x
			\end{aligned}
			\end{equation}
			holds for a.e. $\tau \in (0,T)$. 
		\end{itemize}
	\end{lemma}
	
	\subsection{Limit $\varepsilon \rightarrow 0$}
	
	In order to perform the limit $\varepsilon \rightarrow 0$, we need the following result.
	
	\begin{lemma}
		Let $n\in \mathbb{N}$ be fixed and let $\{ \varrho_{\varepsilon}, \textup{\textbf{u}}_{\varepsilon}, \mathbb{S}_{\varepsilon} \}_{\varepsilon>0}$ be as in Lemma \ref{limit delto to zero}. Moreover, let
		\begin{align*}
		\textbf{f}(\varrho_{\varepsilon})&:= \sqrt{\varepsilon}\  \nabla_x \varrho_{\varepsilon} \\
		\textbf{g}(\varrho_{\varepsilon}, \textup{\textbf{u}}_{\varepsilon})&:= \sqrt{\varepsilon}\  \nabla_x \varrho_{\varepsilon} \cdot \nabla_x  \textup{\textbf{u}}_{\varepsilon}.
		\end{align*}
		Then, passing to a suitable subsequences as the case may be, the following convergences hold as $\varepsilon \rightarrow 0$.
		\begin{align}
		\varrho_{\varepsilon} \overset{*}{\rightharpoonup} \varrho \quad &\mbox{in } L^{\infty}((0,T)\times \Omega), \label{convergence densities epsilon}\\
		\textup{\textbf{u}}_{\varepsilon} \overset{*}{\rightharpoonup} \textup{\textbf{u}} \quad &\mbox{in } L^{\infty}(0,T; W^{1,\infty}(\Omega; \mathbb{R}^d)), \label{convergence velocities epsilon} \\
		\varrho_{\varepsilon}\textup{\textbf{u}}_{\varepsilon} \overset{*}{\rightharpoonup} \varrho \textup{\textbf{u}}\quad &\mbox{in } L^{\infty}((0,T)\times \Omega; \mathbb{R}^d), \label{convergence momenta epsilon} \\
		\varrho_{\varepsilon}\textup{\textbf{u}}_{\varepsilon} \otimes \textup{\textbf{u}}_{\varepsilon} \overset{*}{\rightharpoonup} \varrho \textup{\textbf{u}} \otimes \textup{\textbf{u}} \quad &\mbox{in } L^{\infty}((0,T)\times \Omega; \mathbb{R}^{d\times d}), \label{convergence convective terms epsilon}\\
		\mathbb{S}_{\varepsilon} \rightharpoonup \mathbb{S} \quad &\mbox{in } L^1((0,T) \times \Omega; \mathbb{R}^{d\times d}), \label{convergence viscous stress tensor epsilon}\\ 
		\textbf{f}(\varrho_{\varepsilon})\rightharpoonup \overline{\textbf{f}(\varrho)} \quad &\mbox{in } L^2((0,T) \times \Omega; \mathbb{R}^d), \label{convergence f epsilon}\\
		\textbf{g}(\varrho_{\varepsilon}, \textup{\textbf{u}}_{\varepsilon}) \rightharpoonup \overline{\textbf{g}(\varrho, \textup{\textbf{u}})} \quad &\mbox{in } L^2((0,T) \times \Omega; \mathbb{R}^d). \label{convergence g epsilon}
		\end{align}
	\end{lemma}
	
	\begin{proof}
		From \eqref{energy inequality in epsilon and n} it is easy to deduce the following uniform bounds
		\begin{align}
			\| F (\mathbb{D}_x \textbf{u}_{\varepsilon}) \|_{L^1((0,\infty)\times \Omega)} &\leq c(\overline{E}), \label{estimate F} \\
			\| F^* (\mathbb{S}_{\varepsilon}) \|_{L^1((0,\infty)\times \Omega)} &\leq c(\overline{E}). \label{estimate F*}
		\end{align}
		Similarly to the previous section, from \eqref{estimate F} we obtain
		\begin{equation*}
			\|\textbf{u}_{\varepsilon}\|_{L^q(0,T; W^{1,q}(\Omega; \mathbb{R}^d))}\leq c_1
		\end{equation*}
		for some $q>1$ and a positive constant $c_1$ independent of $\varepsilon>0$, yielding, in view of Lemmas \ref{existence approximated densities}, conditions (ii) and (iii),
		\begin{equation} \label{bound densities}
			e^{-c_1T} \underline{\varrho} \leq \varrho_{\varepsilon}(t,x) \leq e^{c_1T} \overline{\varrho}, \quad \mbox{for all }(t,x) \in [0,T] \times \overline{\Omega}.
		\end{equation}
		We recover convergence \eqref{convergence densities epsilon}. From the energy inequality \eqref{energy inequality in epsilon and n}, it is easy to deduce
		\begin{equation} \label{bound velocities}
		\sup_{t\in [0,T]} \| \textbf{u}_{\varepsilon}(t,\cdot)\|_{W^{1,\infty}(\Omega; \mathbb{R}^d)} \leq c_2,
		\end{equation}
		from which convergence \eqref{convergence velocities epsilon} follows. Combining \eqref{bound densities} and \eqref{bound velocities}, we can recover
		\begin{equation*}
		\varrho_{\varepsilon}\textbf{u}_{\varepsilon} \overset{*}{\rightharpoonup} \textbf{m} \quad \mbox{in } L^{\infty}((0,T)\times \Omega; \mathbb{R}^d).
		\end{equation*}
		Now, notice that \eqref{convergence densities epsilon} can be strengthened to
		\begin{equation*}
		\varrho_{\varepsilon} \rightarrow \varrho \quad \mbox{in } C_{\textup{weak}} ([0,T]; L^p(\Omega)) \quad \mbox{for all } 1<p<\infty
		\end{equation*}
		as $\varepsilon \rightarrow 0$, so that, relaying on the compact Sobolev embedding
		\begin{equation*}
			L^p(\Omega) \hookrightarrow \hookrightarrow W^{-1,1}(\Omega) \quad \mbox{for all } p\geq 1,
		\end{equation*}
		we obtain
		\begin{equation*}
		\varrho_{\varepsilon} \rightarrow \varrho \quad \mbox{in } C([0,T]; W^{-1,1}(\Omega))
		\end{equation*}
		as $\varepsilon \rightarrow 0$. The last convergence combined with \eqref{convergence velocities epsilon}, implies 
		\begin{equation*}
		\textbf{m}= \varrho \textbf{u}  \quad \mbox{a.e. in }(0,T) \times \Omega,
		\end{equation*}
		and thus, we get \eqref{convergence momenta epsilon}. Similarly, from \eqref{convergence velocities epsilon} and \eqref{convergence momenta epsilon} we can deduce \eqref{convergence convective terms epsilon}. Convergence \eqref{convergence viscous stress tensor epsilon} can be deduced from \eqref{estimate F*} using the superlinearity of $F^*$ \eqref{conditions on F*} combined with the De la Vall\'{e}e--Poussin criterion and the Dunford--Pettis theorem. Finally, from \eqref{bound densities} we have in particular that
		\begin{equation*}
			\frac{e^{c_1 T}\overline{\varrho}}{\varrho(t,x)} \geq 1, \quad \mbox{for all } (t,x) \in [0,T] \times \overline{\Omega},
		\end{equation*}
		and thus, from  the energy inequality \eqref{energy inequality in epsilon and n},
		\begin{equation*}
			\varepsilon \int_{0}^{\tau} \int_{\Omega} |\nabla_x \varrho|^2 \  \textup{d}x \textup{d}t \leq  \varepsilon \ e^{c_1 T}\overline{\varrho} \int_{0}^{\tau} \int_{\Omega} P''(\varrho) |\nabla_x \varrho|^2 \ \textup{d}x \textup{d}t \leq c(\overline{\varrho}, T).
		\end{equation*}
		In this way we get \eqref{convergence f epsilon} and, in view of \eqref{bound velocities}, \eqref{convergence g epsilon}.
	\end{proof}
	
	\begin{remark}
		It is worth noticing that the limit density $\varrho$ admits the same upper and lower bounds as in \eqref{bound densities}:
		\begin{equation*} 
		e^{-c_1T} \underline{\varrho} \leq \varrho(t,x) \leq e^{c_1T} \overline{\varrho}, \quad \mbox{for all }(t,x) \in [0,T] \times \overline{\Omega}.
		\end{equation*}
	\end{remark}
	
	We are now ready to let $\varepsilon \rightarrow 0$ in the weak formulations \eqref{weak formulation continuity equation epsilon}, \eqref{weak formulation momentum equation epsilon}; notice in particular that, in view of \eqref{convergence g epsilon}, for any $\tau \in [0,T]$ and any $\bm{\varphi} \in C^1([0,T]; X_n)$
	\begin{equation*}
	\varepsilon \int_{0}^{\tau}\int_{\Omega} \nabla_x \varrho\cdot \nabla_x \textbf{u} \cdot \bm{\varphi} \ \textup{d}x\textup{d}t = \sqrt{\varepsilon} \int_{0}^{\tau}\int_{\Omega} \sqrt{\varepsilon} \ \nabla_x \varrho\cdot \nabla_x \textbf{u} \cdot \bm{\varphi} \ \textup{d}x\textup{d}t \rightarrow 0
	\end{equation*}
	as $\varepsilon \rightarrow 0$.
	
	\begin{lemma} \label{limit epsilon to zero}
		For every fixed $n\in \mathbb{N}$, and any $\varrho_{0,n} \in C(\overline{\Omega})$ such that
		\begin{equation*}
		\int_{\Omega} \left[\frac{1}{2} \frac{|\textup{\textbf{m}}_0|^2}{\varrho_{0,n}} +P(\varrho_{0,n})\right] \textup{d}x \leq \overline{E},
		\end{equation*}
		where the constant $\overline{E}$ is independent of $n$, there exist
		\begin{align*}
			\varrho&=\varrho_n \in L^{\infty}((0,T)\times \Omega), \\
			\textup{\textbf{u}}&=\textup{\textbf{u}}_n \in C([0,T]; X_n),
		\end{align*}
		with
		\begin{equation*}
			\quad e^{-cT} \underline{\varrho} \leq \varrho(t,x) \leq e^{cT} \overline{\varrho}, \quad \mbox{for all }(t,x) \in [0,T] \times \overline{\Omega},
		\end{equation*}
		for a positive constant $c$, such that
		\begin{itemize}
			\item[(i)] the integral identity
			\begin{equation} \label{weak formulation continuity equation n} 
			\left[\int_{\Omega} \varrho \varphi (t, \cdot) \ \textup{d}x\right]_{t=0}^{t=\tau} = \int_{0}^{\tau} \int_{\Omega} (\varrho \partial_t \varphi +\varrho \textup{\textbf{u}}\cdot \nabla_x \varphi ) \ \textup{d} x
			\end{equation}
			holds for any $\tau \in [0,T]$ and any $\varphi \in C^1([0,T]\times \overline{\Omega})$, with $\varrho(0, \cdot)=\varrho_{0,n}$;
			\item[(ii)] there exists
			\begin{equation*}
			\mathbb{S}=\mathbb{S}_n \in L^1((0,T)\times \Omega; \mathbb{R}^{d\times d}_{\textup{sym}})
			\end{equation*}
			such that the integral identity
			\begin{equation} \label{weak formulation balance of momentum n}
			\begin{aligned}
			\left[\int_{\Omega}\varrho \textup{\textbf{u}} \cdot \bm{\varphi}(t, \cdot)\ \textup{d}x\right]_{t=0}^{t=\tau}  &= \int_{0}^{\tau}\int_{\Omega} \left[ \varrho\textup{\textbf{u}} \cdot \partial_t\bm{\varphi}+(\varrho \textup{\textbf{u}} \otimes \textup{\textbf{u}} ): \nabla_x \bm{\varphi} +a\varrho\divv_x \bm{\varphi} \right] \textup{d}x\textup{d}t \\
			&-\int_{0}^{\tau}\int_{\Omega} \mathbb{S}: \nabla_x \bm{\varphi} \ \textup{d}x\textup{d}t 
			\end{aligned}
			\end{equation} 
			holds for any $\tau \in [0,T]$ and any $\bm{\varphi} \in C^1([0,T]; X_n)$, with $(\varrho \textup{\textbf{u}})(0, \cdot)= \textup{\textbf{m}}_0$;
			\item[(iii)] the integral inequality
			\begin{equation} \label{energy inequality in n}
			\int_{\Omega}\left[ \frac{1}{2} \varrho |\textbf{u}|^2 +P(\varrho)\right](\tau, \cdot) \ \textup{d}x + \int_{0}^{\tau} \int_{\Omega}  [F(\mathbb{D}_x \textup{\textbf{u}})+F^*(\mathbb{S})] \  \textup{d}x \textup{d}t \leq \int_{\Omega} \left[\frac{1}{2} \frac{|\textbf{m}_0|^2}{\varrho_{0,n}}+ P(\varrho_{0,n})\right]\textup{d}x
			\end{equation}
			holds for a.e. $\tau \in (0,T)$.
		\end{itemize}
	\end{lemma}
	
	\begin{remark}
		In the energy inequality \eqref{energy inequality in n} we used the lower semi-continuity of the function
		\begin{equation*}
			[\varrho, \textbf{m}] \mapsto \begin{cases}
				0 &\mbox{if }\textbf{m}=0, \\
				\frac{|\textbf{m}|^2}{\varrho} &\mbox{if } \varrho >0, \\
				\infty &\mbox{otherwise},
			\end{cases}
		\end{equation*}
		 and the weak lower semi-continuity in $L^1$ of the functions $F$ and $F^*$, and thus for a.e. $\tau>0$
		\begin{align*}
			\int_{\Omega}\left[ \frac{1}{2} \varrho |\textbf{u}|^2 +P(\varrho)\right](\tau,\cdot) \ \textup{d}x &\leq \liminf_{\varepsilon \rightarrow \infty} \int_{\Omega}\left[ \frac{1}{2} \varrho_{\varepsilon} |\textbf{u}_{\varepsilon}|^2 +P(\varrho)\right](\tau,\cdot) \ \textup{d}x, \\
			\int_{0}^{\tau} \int_{\Omega} [F(\mathbb{D}_x\textbf{u}) +F^*(\mathbb{S})] \ \textup{d}x\textup{d}t &\leq \liminf_{\varepsilon \rightarrow 0}\int_{0}^{\tau} \int_{\Omega} [F(\mathbb{D}_x\textbf{u}_{\varepsilon}) +F^*(\mathbb{S}_{\varepsilon})] \ \textup{d}x\textup{d}t.
		\end{align*}
	\end{remark}
	
	\subsection{Limit $n\rightarrow \infty$} \label{Limit n}
	
	Let $\{\varrho_n, \textbf{m}_n= \varrho_n \textbf{u}_n\}_{n\in \mathbb{N}}$ be the family of approximate solutions obtained in Lemma \ref{limit epsilon to zero}, with correspondent viscous stress tensor $\mathbb{S}_n$. At this stage,  as the initial energies are uniformly bounded by a constant independent of $n$, we can perform the same procedure done in \cite{Bas}, Section 5.1 with $\gamma=1$, to get the following family of convergences as $n \rightarrow \infty$, passing to suitable subsequences as the case may be:
	\begin{align} \label{convergences n to infinity}
		\varrho_n \rightarrow \varrho \quad &\mbox{in } C_{\textup{weak}}([0,T]; L^1(\Omega)), \\
		\textbf{m}_n \rightarrow \textbf{m} \quad &\mbox{in } C_{\textup{weak}}([0,T]; L^1(\Omega; \mathbb{R}^d)), \\
		\textbf{u}_n \rightharpoonup \textbf{u} \quad &\mbox{in } L^q(0,T; W^{1,q}_0(\Omega; \mathbb{R}^d))\\
		\mathbb{S}_n \rightharpoonup \mathbb{S} \quad &\mbox{in } L^1(0,T; L^1(\Omega; \mathbb{R}^{d\times d})), \\
		\mathbbm{1}_{\varrho_n>0}\frac{\textbf{m}_n \otimes \textbf{m}_n}{\varrho_n} \overset{*}{\rightharpoonup} \overline{\mathbbm{1}_{\varrho>0} \frac{\textbf{m}\otimes \textbf{m}}{\varrho}} \quad &\mbox{in }L^{\infty}(0,T; \mathcal{M}(\overline{\Omega}; \mathbb{R}^{d\times d}_{\textup{sym}})).
	\end{align}
	with 
	\begin{equation*}
	\textbf{m} = \varrho \textbf{u} \quad \mbox{a.e. in } (0,T) \times \Omega,
	\end{equation*}
	as a consequence of Lemma 5.2 in \cite{Bas}. 
	
	We are now ready to let $n\rightarrow \infty$ in the weak formulation of the continuity equation \eqref{weak formulation continuity equation n} and the balance of momentum \eqref{weak formulation balance of momentum n}, obtaining that
	\begin{equation*}
		\left[ \int_{\Omega} \varrho \varphi(t,\cdot) \ \textup{d}x \right]_{t=0}^{t=\tau}= \int_{0}^{\tau} \int_{\Omega} [\varrho \partial_t \varphi + \textbf{m}\cdot \nabla_x \varphi] \ \textup{d}x \textup{d}t
	\end{equation*}
	holds for any $\tau\in [0,T]$ and any $\varphi \in C^1([0,T]\times \overline{\Omega})$, with $\varrho(0,\cdot)=\varrho_0$, and
	\begin{equation} \label{weak formulation momentum equation in finite-dimesional space}
	\begin{aligned}
	\left[ \int_{\Omega} \textbf{m}\cdot \bm{\varphi}(t, \cdot) \ \textup{d}x \right]_{t=0}^{t=\tau} &= \int_{0}^{\tau}\int_{\Omega} \left[ \textbf{m} \cdot \partial_t \bm{\varphi} + \mathbbm{1}_{\varrho>0} \frac{\textbf{m}\otimes \textbf{m}}{\varrho}:\nabla_x \bm{\varphi} +a\varrho\divv_x\bm{\varphi}\right] \textup{d}x\textup{d}t \\
	&-\int_{0}^{\tau} \int_{\Omega} \mathbb{S}: \nabla_x \bm{\varphi} \ \textup{d}x\textup{d}t + \int_{0}^{\tau} \int_{\overline{\Omega}} \nabla_x \bm{\varphi} : \textup{d}\mathfrak{R} \ \textup{d}t
	\end{aligned}
	\end{equation}
	holds for any $\tau \in [0,T]$ and any $\bm{\varphi} \in C^1([0,T]; X_n)$, with $n$ arbitrary. As clearly explained by Abbatiello, Feireisl and Novotn\'{y} \cite{AbbFeiNov}, Section 3.4, by a density argument it is possible to extend the validity of the integral identity \eqref{weak formulation momentum equation in finite-dimesional space} for any $\bm{\varphi} \in C^1([0,T] \times \overline{\Omega})$, $\bm{\varphi}|_{\partial \Omega}=0$. Finally, notice that from the energy inequality \eqref{energy inequality in n} we have the following uniform bounds
	\begin{align*}
		\left\| \frac{\textbf{m}_n}{\sqrt{\varrho_n}} \right \| _{L^{\infty}(0,T; L^2(\Omega; \mathbb{R}^d))} &\leq c(\overline{E}),  \\
		\| P(\varrho_n)\|_{L^{\infty}(0,T; L^1(\Omega))} & \leq c(\overline{E}), 
	\end{align*}
	from which it is possible to deduce that
	\begin{align*}
		\frac{|\textbf{m}_n|^2}{\varrho_n} \overset{*}{\rightharpoonup} \overline{\frac{|\textbf{m}|^2}{\varrho}} \quad &\mbox{in } L^{\infty}(0,\infty; \mathcal{M}(\overline{\Omega})) \\
		P(\varrho_n) \overset{*}{\rightharpoonup} \overline{P(\varrho)} \quad &\mbox{in } L^{\infty}(0,\infty; \mathcal{M}(\overline{\Omega}))
	\end{align*}
	as $n \rightarrow \infty$. Thus,
	\begin{equation*}
		\mathfrak{R} \in L^{\infty}_{\textup{weak}}(0,T; \mathcal{M}^+(\overline{\Omega}; \mathbb{R}^{d\times d}_{\textup{sym}}))
	\end{equation*}
	appearing in \eqref{weak formulation momentum equation in finite-dimesional space} has been chosen in such a way that
	\begin{equation*} 
		\textup{d}\mathfrak{R} =   \left(\overline{\mathbbm{1}_{\varrho>0}\frac{\textbf{m}\otimes \textbf{m}}{\varrho}} -\mathbbm{1}_{\varrho>0}\frac{\textbf{m}\otimes \textbf{m}}{\varrho}\right) \textup{d}x +\psi(t)\mathbb{I},
	\end{equation*}
	where the time-dependent function $\psi$ is chosen in such a way to guarantee 
	\begin{equation*}
		\frac{1}{\lambda} \ \textup{d}  \trace[\mathfrak{R}] = \frac{1}{2}\left( \overline{\frac{|\textbf{m}|^2}{\varrho}} - \frac{|\textbf{m}|^2}{\varrho} \right) \textup{d}x + \left( \overline{P(\varrho)}- P(\varrho) \right) \textup{d}x
	\end{equation*}
	for a.e. $\tau \in (0,T)$; see \cite{Bas}, Section 5.4 for further details.
	
	We proved the following result.
	\begin{theorem} \label{existence dissipative solutions}
		For every fixed initial data 
		\begin{equation*}
			[\varrho_{0}, \textup{\textbf{m}}_0] \in L^1(\Omega) \times L^1(\Omega; \mathbb{R}^d), 
		\end{equation*}
		with 
		\begin{equation} 
		\int_{\Omega} \left[ \frac{1}{2} \frac{|\textup{\textbf{m}}_0|^2}{\varrho_0} + \varrho_{0} \log \varrho_{0}\right] \textup{d}x <\infty,
		\end{equation}
		problem \eqref{continuity equation}--\eqref{initial conditions} admits a dissipative solution in the sense of Definition \ref{dissipative solution}.
	\end{theorem}
	
	\section{Existence of weak solutions} \label{Existence of weak solutions}
	
	Choosing $q> d$ in \eqref{relation F and trsce-less part}, we get the existence of weak solutions to models describing a general viscous compressible fluid \eqref{continuity equation}--\eqref{initial conditions}, or equivalently, the Reynold stress $\mathfrak{R}$ appearing in Definition \ref{dissipative solution} is identically zero. In particular, we improve the work by Matu\v{s}$\mathring{\mbox{u}}$-Ne\v{c}asov\'{a} and Novotn\'{y} \cite{MatNov}, where existence was achieved in the framework of measure-valued solutions.
	
	We can repeat the same procedure performed in the previous section until we get to Lemma \ref{limit epsilon to zero}. We can now prove the following crucial result.
	\begin{lemma}
		Let $q> d$ in \eqref{relation F and trsce-less part} and let $\{\varrho_n, \textup{\textbf{m}}_n= \varrho_n \textup{\textbf{u}}_n\}_{n\in \mathbb{N}}$ be the family of approximate solutions obtained in Lemma \ref{limit epsilon to zero}. Then, passing to a suitable subsequence as the case may be,
		\begin{equation} \label{weak convergence of convective terms in L1}
		\varrho_n \textup{\textbf{u}}_n \otimes \textup{\textbf{u}}_n \rightharpoonup \varrho \textup{\textbf{u}}\otimes \textup{\textbf{u}} \quad \mbox{in } L^1((0,T)\times \Omega; \mathbb{R}^{d\times d})
		\end{equation}
		as $n \rightarrow \infty$.
	\end{lemma}
	\begin{proof}
		Proceeding as in \cite{Bas}, Sections 5.1 and 5.2, we have
		\begin{align*}
			\varrho_n \rightarrow \varrho \quad &\mbox{in } C_{\textup{weak}}([0,T]; L^1(\Omega)), \\
			\varrho_n \textbf{u}_n \rightarrow \varrho\textbf{u} \quad &\mbox{in } C_{\textup{weak}}([0,T]; L^1(\Omega; \mathbb{R}^d))
		\end{align*}
		as $n \rightarrow \infty$, where the sequence $\{ \varrho_n \textbf{u}_n(t,\cdot) \}_{n \in \mathbb{N}}$ is equi-integrable in $L^1(\Omega; \mathbb{R}^d)$ for a.e. $t\in (0,T)$. Thanks to the slightly modified De la Vall\'{e}e--Poussin criterion, which we report in the Appendix, Theorem \ref{De la Valee Poussin criterion}, there exists a Young function $\Psi$ satisfies the $\Delta_2$-condition \eqref{delta 2 condition} such that 
		\begin{equation*}
			\varrho_n \textbf{u}_n \overset{*}{\rightharpoonup} \varrho \textbf{u} \quad \mbox{in } L^{\infty}(0,T; L_{\Psi}(\Omega; \mathbb{R}^d)),
		\end{equation*}
		Moreover, due to the compact Sobolev embedding
		\begin{equation*}
			L^p(\Omega)\hookrightarrow \hookrightarrow W^{-1,q'}(\Omega) \quad \mbox{for any }p\geq 1,
		\end{equation*}
		 which is true since $q>d$ from our hypothesis, we can prove that the sequence $\{ \varrho_n \textbf{u}_n \otimes \textbf{u}_n \}_{n\in \mathbb{N}}$ is equi-integrable in $L^1((0,T)\times \Omega; \mathbb{R}^{d\times d})$. Indeed, let $\varepsilon>0$ be fixed and let the constant $c>0$ be such that
		\begin{equation*}
		\| \textbf{u}_n \|_{L^q(0,T; W^{1,q}(\Omega; \mathbb{R}^d))} \leq c,
		\end{equation*}
		uniformly in $n$. Let $\widetilde{\varepsilon}=\widetilde{\varepsilon}(\varepsilon) >0$ be chosen in such a way that
		\begin{equation*}
		\widetilde{\varepsilon} < \left(c \ T^{\frac{1}{q'}}\right)^{-1} \varepsilon .
		\end{equation*}
		From the equi-integrability of the sequence $\{ \varrho_n \textbf{u}_n\}_{n\in \mathbb{N}}$, there exists $\delta=\delta(\widetilde{\varepsilon})>0$ such that
		\begin{equation*}
		\int_{M} |\varrho_n \textbf{u}_n|(t) \ \textup{d}x < \widetilde{\varepsilon}, \quad \mbox{for every } M\subset \Omega \mbox{ s.t. } |M| <\delta,
		\end{equation*}
		for every $n\in \mathbb{N}$. Let $(t_1, t_2) \times M \subset [0,T] \times \Omega$ such that
		\begin{equation*}
		|(t_1, t_2) \times M| < \delta.
		\end{equation*}
		Then, for every $n\in \mathbb{N}$,
		\begin{align*}
		\int_{t_1}^{t_2} \int_{M} |\varrho_n \textbf{u}_n \otimes \textbf{u}_n| \ \textup{d}x \textup{d}t &\leq \int_{0}^{T} \int_{M} |\varrho_n \textbf{u}_n \otimes \textbf{u}_n| \ \textup{d}x \textup{d}t \\
		&\leq \| \varrho_n \textbf{u}_n \|_{L^{q'}(0,T; L^1(M))} \| \textbf{u}_n \|_{L^q(0,T; W^{1,q}(M))} \\
		&\leq c \left[  \int_{0}^{T} \left(\int_{M} |\varrho_n \textbf{u}_n|(t) \ \textup{d}x\right)^{q'} \textup{d}t\right]^{\frac{1}{q'}} \\
		&\leq c \ \widetilde{\varepsilon} \ T^{\frac{1}{q'}} \\
		&< \varepsilon.
		\end{align*}
		Consequently, we can adapt Lemma 5.2 in \cite{Bas} replacing the sequence of densities $\{\varrho_n\}_{n\in \mathbb{N}}$ with the sequence of momenta $\{\varrho_n \textbf{u}_n\}_{n\in \mathbb{N}}$ to obtain \eqref{weak convergence of convective terms in L1}.
	\end{proof}
	
	Letting $n\rightarrow \infty$ in the weak formulation of the continuity equation \eqref{weak formulation continuity equation n} and the balance of momentum \eqref{weak formulation balance of momentum n}, we obtain the following result. 
	
	\begin{theorem} \label{existence weak solutions}
		Let $q>d$ in \eqref{relation F and trsce-less part}. For every fixed initial data 
		\begin{equation*}
		[\varrho_{0}, \textup{\textbf{m}}_0] \in L^1(\Omega) \times L^1(\Omega; \mathbb{R}^d), 
		\end{equation*}
		with 
		\begin{equation} 
			\int_{\Omega} \left[ \frac{1}{2} \frac{|\textup{\textbf{m}}_0|^2}{\varrho_0} + \varrho_{0} \log \varrho_{0}\right] \textup{d}x < \infty,
		\end{equation}
		problem \eqref{continuity equation}--\eqref{initial conditions} admits a weak solution
		\begin{equation*}
			[\varrho, \varrho\textup{\textbf{u}}] \in C_{\textup{weak}}([0,T]; L^1(\Omega)) \times  C_{\textup{weak}}([0,T]; L^1(\Omega;\mathbb{R}^d)),
		\end{equation*}
		meaning that the following holds.
		\begin{itemize}
			\item[(i)]  $\varrho \geq 0$ in $(0,T) \times \Omega$.
			\item[(i)] The integral identity
			\begin{equation*} 
			\left[ \int_{\Omega} \varrho \varphi(t,\cdot) \ \textup{d}x \right]_{t=0}^{t=\tau}= \int_{0}^{\tau} \int_{\Omega} [\varrho \partial_t \varphi + \varrho \textup{\textbf{u}}\cdot \nabla_x \varphi] \ \textup{d}x \textup{d}t
			\end{equation*}
			holds for any $\tau \in [0,T]$ and any $\varphi \in C_c^1([0,T]\times \overline{\Omega})$, with $\varrho(0,\cdot)=\varrho_0$.
			\item[(iii)] There exists
			\begin{equation*}
			\mathbb{S} \in L^1 (0,T; L^1(\Omega; \mathbb{R}^{d\times d}_{\textup{sym}}))
			\end{equation*}
			such that the integral identity
			\begin{equation*}
			\begin{aligned}
			\left[ \int_{\Omega} \varrho\textup{\textbf{u}}\cdot \bm{\varphi}(t, \cdot) \ \textup{d}x \right]_{t=0}^{t=\tau} &= \int_{0}^{\tau}\int_{\Omega} \left[ \varrho\textup{\textbf{u}} \cdot \partial_t \bm{\varphi} + \varrho \textup{\textbf{u}}\otimes \textup{\textbf{u}}:\nabla_x \bm{\varphi} +a\varrho\divv_x\bm{\varphi}\right] \ \textup{d}x \textup{d}t \\
			&- \int_{0}^{\tau} \int_{\Omega} \mathbb{S}: \nabla_x \bm{\varphi} \ \textup{d}x \textup{d}t 
			\end{aligned}
			\end{equation*}
			holds for any $\tau \in [0,T]$ and any $\bm{\varphi} \in C^1_c([0,T]\times \overline{\Omega}; \mathbb{R}^d)$, $\bm{\varphi}|_{\partial \Omega}=0$, with $(\varrho\textup{\textbf{u}})(0,\cdot)=\textup{\textbf{m}}_0$.
			\item[(iv)] the energy inequality
			\begin{equation*} 
			\begin{aligned}
			\int_{\Omega} \left[ \frac{1}{2} \frac{|\textbf{m}|^2}{\varrho} + a\varrho \log \varrho \right](\tau,\cdot) \ \textup{d}x &+ \int_{0}^{\tau} \int_{\Omega} \left[ F(\mathbb{D}\textbf{u})+ F^*(\mathbb{S})\right] \ \textup{d}x \textup{d}t \\
			&\leq \int_{\Omega} \left[ \frac{1}{2} \frac{|\textbf{m}_0|^2}{\varrho_0} + a\varrho_0 \log \varrho_0 \right]  \textup{d}x 
			\end{aligned}
			\end{equation*}
			holds for a.e. $\tau  \in (0,T)$.
		\end{itemize}
	\end{theorem}

	\appendix
	
	\section{De la Vall\'{e}e--Poussin criterion} \label{De la Vallee-Poussin criterion}
	
	In this section, we prove a slightly modified version of the De la Vall\'{e}e--Poussin criterion as we require the stronger condition, with respect to the standard formulation, that the Young function satisfies the $\Delta_2$-condition. We first recall the definitions of Young function and $\Delta_2$-condition.
	\begin{definition}
		\begin{itemize}
			\item[(i)] We say that $\Phi$ is a \textit{Young function} generated by $\varphi$ if
			\begin{equation*}
			\Phi(t) = \int_{0}^{t} \varphi (s) \ \textup{d}s \quad \mbox{for any } t\geq 0,
			\end{equation*}
			where the real-valued function $\varphi$ defined on $[0,\infty)$ is non-negative, non-decreasing, left-continuous and such that 
			\begin{equation*}
			\varphi(0)=0, \quad \lim_{s\rightarrow \infty}\varphi(s)=\infty.
			\end{equation*}
			\item[(ii)] A Young function $\Phi$ is said to satisfy the $\Delta_2$-\textit{condition} if there exist a positive constant $K$ and $t_0 \leq 0$ such that
			\begin{equation} \label{delta 2 condition}
			\Phi(2t) \leq K \Phi (t) \quad \mbox{for any } t\geq t_0.
			\end{equation}
		\end{itemize}
	\end{definition}
	
	\begin{theorem} \label{De la Valee Poussin criterion}
		Let $Q \subset \mathbb{R}^d$ be a bounded measurable set and let $\{ f_n \}_{n\in \mathbb{N}}$ be a sequence in $L^1(Q)$. Then, the following statements are equivalent. 
		\begin{itemize}
			\item[(i)] The sequence $\{ f_n \}_{n\in \mathbb{N}}$ is equi-integrable, meaning that for any $\varepsilon>0$ there exists $\delta=\delta(\varepsilon)>0$ such that
			\begin{equation*}
			\int_{M} |f_n (y)| \ \textup{d}y < \varepsilon \quad \mbox{for any } M\subset Q \mbox{ such that } |M| < \delta,
			\end{equation*}
			independently of $n$.
			\item[(ii)] There exists a Young function $\Phi$ satisfying the $\Delta_2$-condition \eqref{delta 2 condition} such that the sequence $\{ f_n \}_{n\in \mathbb{N}}$ is uniformly bounded in the Orlicz space $L_{\Phi}(Q)$.
		\end{itemize}
	\end{theorem} 
	\begin{proof}
		(ii) $\Rightarrow$ (i) See Pedregal  \cite{Ped}, Chapter 6, Lemma 6.4. 
		
		(i) $\Rightarrow$ (ii) For $n\in \mathbb{N}$ and $j\geq 1$ fixed, let
		\begin{equation*}
		\mu_j(f_n):= |\{ y\in Q: \ |f_n(y)| >j \}|.
		\end{equation*}
		As the sequence $\{ f_n \}_{n\in \mathbb{N}}$ is equi-integrable, from the Dunford-Pettis theorem there exists a strictly increasing sequence of positive integers $\{ C_m \}_{m\in \mathbb{N}}$ such that for each $m$
		\begin{equation*}
		\sup_{n\in \mathbb{N}} \int_{\{ |f_n|>C_m\}} |f_n(y)| \ \textup{d}y \leq \frac{1}{2^m}.
		\end{equation*}
		For $n\in \mathbb{N}$ and $m\geq 1$ fixed
		\begin{equation*}
		\int_{\{ |f_n|>C_m\}} |f_n(y)| \ \textup{d}y = \sum_{j=C_m}^{\infty} \int_{\{ j<|f_n|\leq j+1\}} |f_n(y)| \ \textup{d}y \geq \sum_{j=C_m}^{\infty} j \ [\mu_j(f_n)-\mu_{j+1}(f_n)] \geq \sum_{j=C_m}^{\infty}\mu_j(f_n).
		\end{equation*}
		In particular, we obtain
		\begin{equation*}
		\sum_{m=1}^{\infty} \sum_{j=C_m}^{\infty}\mu_j(f_n) \leq \sum_{m=1}^{\infty} \int_{\{ |f_n|>C_m\}} |f_n(y)| \ \textup{d}y \leq \sum_{m=1}^{\infty}  \frac{1}{2^m} =1.
		\end{equation*}
		
		For $m\geq 0$, we define
		\begin{equation*}
		\alpha_m= \begin{cases}
		0 &\mbox{if } m<C_1, \\
		\max \{k: \ C_k \leq m\} &\mbox{if } m\geq C_1.
		\end{cases}
		\end{equation*}
		Notice that
		\begin{equation} \label{relation alpha and C}
		\alpha_m \geq j \quad \Leftrightarrow \quad C_j \leq m.
		\end{equation}
		It is straightforward that $\alpha_m \rightarrow \infty$ as $m\rightarrow \infty$. We define a step function $\varphi$ on $[0,\infty)$ by
		\begin{equation*}
		\varphi(s) = \sum_{m=0}^{\infty} \alpha_m \chi_{(m,m+1]}(s) \quad \mbox{for any }0\leq s<\infty.
		\end{equation*}
		It is clear that $\varphi$ is non-negative, non-decreasing, left-continuous and such that $\varphi(0)=0$, $\lim_{s\rightarrow \infty} \varphi(s)= \infty$. Then, we can define the Young function $\Phi$ generated by $\varphi$ as
		\begin{equation*}
		\Phi(t)= \int_{0}^{t} \varphi(s)\ \textup{d}s, \quad \mbox{for any }0\leq t< \infty.
		\end{equation*}
		At this point, notice that we have the freedom to take the constants $C_j$, $j\geq 1$, as large as we want and consequently, the constants $\alpha_m$, $m\geq 1$, will be as small as we want. More precisely, we may find a positive constant $c$ such that
		\begin{equation*}
		\alpha_{2m} \leq c \ \alpha_m \quad \mbox{for any } m\geq 1.
		\end{equation*}
		We then obtain, for all $s\in [0,\infty)$,
		\begin{equation*}
		\varphi(2s) = \sum_{m=0}^{\infty} \alpha_m \chi_{\left(\frac{m}{2}, \frac{m+1}{2}\right)}(s) = \sum_{k=0}^{\infty} \alpha_{2k} \chi_{\left(k, k+ \frac{1}{2}\right)}(s) \leq c \sum_{k=0}^{\infty} \alpha_k \chi_{\left(k, k+ \frac{1}{2}\right)}(s) \leq c\ \varphi(s);
		\end{equation*}
		consequently, for all $t\in [0,\infty)$,
		\begin{equation*}
		\Phi(2t) = \int_{0}^{2t} \varphi(s) \ \textup{d}s = 2\int_{0}^{t} \varphi(2z) \ \textup{d}z \leq 2c \int_{0}^{t} \varphi(z) \ \textup{d}z = 2c \ \Phi(t),
		\end{equation*}
		and thus we get that the Young function $\Phi$ satisfies the $\Delta_2$-condition \eqref{delta 2 condition}. 
		
		Finally, for $n\in \mathbb{N}$ fixed, using the fact that $\Phi(0)=\Phi(1)=0$ and for $j\geq 1$, noticing that $\alpha_0=0$,
		\begin{equation*}
		\Phi(j+1) = \int_{0}^{j+1} \varphi(s) \ \textup{d} s = \sum_{m=0}^{j} \int_{m}^{m+1} \varphi(s) \ \textup{d} s \leq \sum_{m=0}^{j} \varphi(m+1) = \sum_{m=0}^{j} \alpha_m= \sum_{m=1}^{j} \alpha_m,
		\end{equation*}
		we get 
		\begin{align*}
		\int_{Q} \Phi(|f_n(y)|) \ \textup{d} y  &= \int_{\{ |f_n|=0 \}} \Phi (|f_n(y)|) \ \textup{d} y + \sum_{j=0}^{\infty} \int_{\{ j<|f_n|\leq j+1\}} \Phi(|f_n(y)|) \ \textup{d}y \\
		&\leq \sum_{j=1}^{\infty} [\mu_j(f_n)-\mu_{j+1}(f_n)] \ \Phi(j+1) \\
		&\leq \sum_{j=1}^{\infty} [\mu_j(f_n)-\mu_{j+1}(f_n)]  \sum_{m=1}^{j} \alpha_m \\
		&= \sum_{m=1}^{\infty} \alpha_m \sum_{j=m}^{\infty} [\mu_j(f_n)-\mu_{j+1}(f_n)] \\
		&=  \sum_{m=1}^{\infty} \alpha_m \mu_m(f_n)= \sum_{m=1}^{\infty} \mu_m(f_n) \sum_{j=1}^{\alpha_m} 1 = \sum_{j=1}^{\infty} \sum_{m=C_j}^{\infty} \mu_m(f_n) \leq 1
		\end{align*}
		where we used \eqref{relation alpha and C} in the last line. In particular, we obtain that the sequence $\{ f_n \}_{n\in \mathbb{N}}$ is uniformly bounded in the Orlicz space $L_{\Phi}(Q)$.
	\end{proof}

	\bigskip
	
	\centerline{\bf Acknowledgement}
	
	This work was supported by the Einstein Foundation, Berlin. The author wishes to thank her advisor Prof. Eduard Feireisl for the helpful advice and discussions.


\begin{thebibliography}{}
		
		\bibitem{AbbFeiNov}
		A. Abbatiello, E. Feireisl and A. Novotn\'{y},
		\textit{Generalized solutions to mathematical models of compressible viscous fluids}, Discrete \& Continuous Dynamical Systems \textbf{41}(1): 1--28; 2021
		
		\bibitem{Bas}
		D. Basari\'{c},
		\textit{Semiflow selection to models of general compressible viscous fluids}, J. Math. Fluid Mech. \textbf{23}(2); 2021 
		
		\bibitem{BleMalRaj}
		J. Blechta, J. M\'{a}lek and J. R. Rajagopal,
		\textit{On the classification of incompressible fluids and a mathematical analysis of the equations that govern their motion}, 	arXiv:1902.04853; 2019
		
		\bibitem{BreCiaDie}
		D. Breit, A. Cianchi and L. Diening,
		\textit{Trace--free Korn inequality in Orlicz spaces}, SIAM J. Math. Anal. \textbf{49}(4): 2496--2526; 2017
		
		\bibitem{ChaJinNov}
		T. Chang, B. J. Jin and A. Novotn\'{y}, \textit{Compressible Navier-Stokes system with inflow-outflow boundary data}, SIAM J. Math. Anal. \textbf{51}(2): 1238--1278; 2019
		
		\bibitem{CriDonSpi}
		G. Crippa, C. Donadello and L. V. Spinolo, 
		\textit{A note on the initial-boundary value problem for continuity equations with rough coefficients}, HYP 2012 conference proceedings, AIMS Series in Appl. Math. \textbf{8}: 957--966; 2014
		
		\bibitem{Fei}
		E. Feireisl, \textit{Dynamics of viscous compressible fluids}, Oxford University Press, Oxford; 2003
		
		\bibitem{FeiLiaMal}
		E. Feireisl, X. Liao and J. M\'{a}lek,
		\textit{Global weak solutions to a class of non-Newtonian compressible fluids}, Math. Meth. Appl. Sci. \textbf{38}(16): 3482--3494; 2015
		
		\bibitem{Gir}
		V. Girinon,
		\textit{Navier-Stokes equations with nonhomogeneous boundary conditions in a bounded three-dimensional domain}, J. Math. Fluid Mech. \textbf{13}: 309--339; 2011
		
		\bibitem{Lio}
		P. L. Lions,
		\textit{Mathematical topics in fluid mechanics, Volume 2: compressible models}, Oxford Science Publications, Oxford; 1998
		
		\bibitem{Mam}
		A. E. Mamontov,
		\textit{Global solvability of the multidimensional Navier-Stokes equations of a compressible fluid with nonlinear viscosity. I}, Sib. Math. J. \textbf{40}: 351--362; 1999
		
		\bibitem{Mam1}
		A. E. Mamontov,
		\textit{Global solvability of the multidimensional Navier-Stokes equations of a compressible fluid with nonlinear viscosity. II}, Sib. Math. J. \textbf{40}: 541--555; 1999
		
		\bibitem{MatNov}
		\v{S}. Matu\v{s}\accent23u-Ne\v{c}asov\'{a} and A. Novotn\'{y},
		\textit{Measure--valued solution for non--Newtonian compressible isothermal monopolar fluid}, Acta Applicandae Mathematica \textbf{37}: 109--128; 1994
		
		\bibitem{Ped}
		P. Pedregal,
		\textit{Parametrized measures and variational principles}, Birkha\"{u}ser, Basel; 1997
		
	\end{thebibliography}
\end{document}